\documentclass[12pt]{article}
\usepackage[margin=2.5 cm]{geometry}
\usepackage{natbib}
\usepackage{amsmath}
\usepackage{amssymb}
\usepackage[titletoc]{appendix}

\usepackage{cancel}
\usepackage{graphicx}
\usepackage{latexsym}

\usepackage{pdflscape}

\usepackage{mathrsfs}
\usepackage{mathtools}
\usepackage{subfig}
\usepackage{thmtools,thm-restate}
\usepackage{appendix}
\usepackage{mathrsfs}

\usepackage{enumerate}
\usepackage{mathtools}
\usepackage{graphicx}
\usepackage{setspace}
\usepackage{mdframed}
\usepackage{enumitem, color, amssymb}
\usepackage{multicol}
\usepackage{subfig}

\usepackage{subfig}
\usepackage{wrapfig}
\usepackage{tikz}
\usetikzlibrary{shapes.geometric}
\usetikzlibrary{decorations.pathreplacing}
\usetikzlibrary{positioning}

\definecolor{presentBlue}{RGB}{51,51,178}
\definecolor{baseBlue}{RGB}{0,77,77}
\definecolor{varBlue1}{RGB}{31,118,118}
\definecolor{varBlue2}{RGB}{13,97,97}
\definecolor{varBlue3}{RGB}{0,54,54}
\definecolor{varBlue1}{RGB}{0,30,30}

\definecolor{presentGreen}{RGB}{0,102,0}
\definecolor{varGreen1}{RGB}{41,158,41}
\definecolor{varGreen2}{RGB}{18,129,18}
\definecolor{varGreen3}{RGB}{0,72,0}
\definecolor{varGreen4}{RGB}{0,102,0}

\definecolor{baseRed}{RGB}{128,0,0}
\definecolor{varRed1}{RGB}{197,51,51}
\definecolor{varRed2}{RGB}{161,22,22}
\definecolor{varRed3}{RGB}{90,0,0}
\definecolor{varRed4}{RGB}{50,0,0}

\definecolor{presentGray}{RGB}{63,63,63}

\usepackage{pf2}

\usepackage{subfig}
\usepackage{wrapfig}

\usepackage{epstopdf}
\usepackage[T1]{fontenc}
\usepackage{tabu}

\usepackage[T1]{fontenc}
\usepackage[]{algorithm2e} 
\newcommand{\wh}[1]{\widehat{#1}}

\newcommand{\ov}[1]{\overline{#1}}
\newcommand{\mb}[1]{\mathbf{#1}}

\newcommand{\mbU}{\mb{U}}
\newcommand{\mbV}{\mb{V}}

\newcommand{\argmin}{\mathrm{argmin}}

\newcommand{\natD}{\mb{d}_{\natural}}

\newcommand{\natmbU}{\natural(\mbU)}
\newcommand{\natmbV}{\natural(\mbV)}
\newcommand{\RO}{\mb{RO}}

\mathchardef\mhyphen="2D

\newcommand{\cl}{\mathrm{cl}}
\newcommand{\interior}{\mathrm{int}}
\newcommand{\bd}{\mathrm{bd}}
\newcommand{\ext}{\mathrm{ext}}
\newcommand{\opt}{\scriptstyle \mathrm{opt}}

\newcommand{\seq}[3]{ \left\{{#1}\right\}_{#2}^{#3} }

\mathchardef\mhyphen="2D

\newtheorem{theorem}{Theorem}[section]
\newtheorem{lemma}[theorem]{Lemma}
\newtheorem{remark}[theorem]{Remark}

\newtheorem{definition}[theorem]{Definition}
\newtheorem{example}[theorem]{Example}

\usepackage{graphicx}
\usepackage{dcolumn}
\usepackage{bm}

\usepackage{titlesec}

\titleformat*{\section}{\large\bfseries}
\titleformat*{\subsection}{\large\bfseries}
\titleformat*{\subsubsection}{\normalsize\bfseries}
\titleformat*{\paragraph}{\normalsize\bfseries}
\titleformat*{\subparagraph}{\normalsize\bfseries}

\usepackage{tikz}
\usetikzlibrary{shapes.geometric}
\usetikzlibrary{decorations.pathreplacing}
\usetikzlibrary{positioning}

\definecolor{presentBlue}{RGB}{51,51,178}
\definecolor{baseBlue}{RGB}{0,77,77}
\definecolor{varBlue1}{RGB}{31,118,118}
\definecolor{varBlue2}{RGB}{13,97,97}
\definecolor{varBlue3}{RGB}{0,54,54}
\definecolor{varBlue4}{RGB}{0,30,30}

\definecolor{presentGreen}{RGB}{0,102,0}
\definecolor{varGreen1}{RGB}{41,158,41}
\definecolor{varGreen2}{RGB}{18,129,18}
\definecolor{varGreen3}{RGB}{0,72,0}
\definecolor{varGreen4}{RGB}{0,102,0}

\definecolor{baseRed}{RGB}{128,0,0}
\definecolor{varRed1}{RGB}{197,51,51}
\definecolor{varRed2}{RGB}{161,22,22}
\definecolor{varRed3}{RGB}{90,0,0}
\definecolor{varRed1}{RGB}{50,0,0}

\definecolor{presentGray}{RGB}{63,63,63}

\title{Stability and Continuity in Robust Linear and Linear Semi-Infinite Optimization}
\author{Timothy C.Y. Chan, University of Toronto \and Philip Allen Mar\thanks{Corresponding author, email: philip.mar@mail.utoronto.ca}, University of Toronto}

\begin{document}

\maketitle

\begin{abstract}%
We consider the stability of Robust Optimization problems with respect to perturbations in their uncertainty sets. We focus on Linear Optimization problems, including those with a possibly infinite number of constraints, also known as Linear Semi-Infinite Optimization (LSIO) problems, and consider uncertainty in both the cost function and constraints. We prove Lipschitz continuity of the optimal value and $\epsilon$-approximate optimal solution set with respect to the Hausdorff distance between uncertainty sets and with an explicit Lipschitz constant that can be calculated. In addition, we prove closedness and upper semi-continuity for the optimal solution set mapping with respect to the uncertainty set.

\end{abstract}%

\section{Introduction}

The theory and applications of Robust Optimization (RO) have grown rapidly over the past two decades. However, one important theoretical property of optimization problems, \emph{stability}, has been largely unexplored in RO with a few notable exceptions. The stability of the optimal value and the optimal solution set, with respect to certain problem parameters, are examples of the types of stability that are of particular interest in optimization problems. For RO problems, one natural quantity that stability can be defined with respect to is the \emph{uncertainty set}, which is our focus in this paper.


Before proceeding further, we first make a distinction between what we call \emph{quantitative} and \emph{qualitative} stability.  Quantitative stability properties are those that are endowed with an explicit constant that describes the degree of stability.  One such example is the Lipschitz continuity of the optimal value with some constant $L$, which bounds the change in the optimal value by the product of $L$ and some amount of perturbation in the parameters. On the other hand, qualitative stability properties are those that do not furnish an explicit constant measuring the degree of stability.  One example is upper semi-continuity of the optimal solution set, which describes the limiting behavior of the optimal solution set without providing any information about the rate of convergence.

Some RO research has considered stability properties through either a sensitivity analysis or problem well-posedness viewpoint. \citet{RobustConvexOpt} proved a bound on the difference in the optimal value between a nominal problem and its robust counterpart, with the bound depending on the specific uncertainty set in the robust problem. \citet{el1998robust} proved H{\" o}lder continuity of the approximate optimal solution set in semidefinite programming with respect to perturbations in the linear matrix inequality coefficients. More recently, \citet{crespi2015quasiconvexity} derived closedness properties for the optimal solution set of a robust vector optimization problem. In these papers, the stability results are with respect to a fixed uncertainty set.


Several recent papers have proved convergence properties of RO problems with respect to changes in the uncertainty set.  In particular, it was shown that the optimal solution set (and value) of the robust problem converged to the optimal solution set (and value) of the nominal problem as a polyhedral uncertainty set converged to a singleton in robust linear programming (\citet{AdaptiveRobustIMRT}). An analogous result was shown for the case of general convex uncertainty sets in robust conic programming (\citet{werner2010consistency}). These results have been extended to the case where a regularized convex uncertainty set converges to a non-singleton set in robust quadratic programming (\citet{moazeni2013regularized}).

Stability has been well-studied in the related field of Stochastic Programming (see \citet{schultz2000some} for an extensive literature review). In particular, there exist results regarding Lipschitz continuity of the optimal value (\citet{liu2013stability}), of the optimal solution sets (\citet{romisch1996lipschitz,shapiro1994quantitative,liu2013stability}) and approximate optimal solution sets (\citet{romisch2007stability}), with respect to the probability measure. Additionally, recent research in Distributionally Robust Stochastic Programming (DRSP) or Distributionally Robust Optimization (DRO) stability include convergence of optimal values (and optimal solution sets) as the ambiguity sets converge, focusing primarily on uncertainty in the cost function (\citet{sun2015convergence,DistributionalInterpretation,dupavcova2011uncertainties}). Quantitative convergence of the optimal value has also been established for a two-stage DRSP problem, including uncertainty in the constraints (\citet{riis2005applying}).


Stability has also been studied in Linear Semi-Infinite Optimization (LSIO) problems, which include the class of robust convex optimization problems (\citet{RobustConvexOpt}). In particular, \citet{canovas2006lipschitz} showed Lipschitz continuity of the optimal value while \citet{canovas1999stability} showed convergence of the optimal solution set. Even though RO problems can be viewed as LSIO problems, the above-mentioned properties do not directly yield stability properties for RO due to the difference in their problem structures. In these LSIO results, the LSIO problems have constraints that are indexed relative to a pre-defined and common index set.  On the other hand, RO problems have constraints that are typically defined by an uncertainty set, which may not constitute a common index set with other RO problems. However, as we will show, stability properties for RO can be established by taking advantage of the lack of explicit indexing in RO constraints, leveraging existing LSIO results.



In this paper, we present a comprehensive study of stability, including both quantitative and qualitative properties, in robust linear optimization with respect to the uncertainty set.  In particular, we provide quantitative stability results for the optimal value and $\epsilon$-approximate optimal solution set (hereafter abbreviated as $\epsilon$-optimal solution set), and qualitative stability results for the optimal solution set. Our results hold for uncertainty in both the cost function and constraints. Our results also extend to the case where the nominal problem is a LSIO problem itself, and thus, we establish novel stability results for robust LSIO. Because of the connection to robust LSIO, our stability results serve as a stepping stone to understanding stability in a general class of robust convex optimization problems. In addition to minor technical assumptions, the primary assumptions we make about the underlying RO problem are that its uncertainty set is convex and compact, that the strong Slater condition holds, and that its optimal solution set is non-empty and bounded.



The specific contributions of this paper are:


\begin{enumerate}
\item \textbf{Lipschitz continuity of the optimal value}. This is the first quantitative stability result for the optimal value in RO and robust LSIO with uncertainty in the constraints, providing an explicit Lipschitz constant that can be calculated from the problem parameters. The quantitative nature of this result contrasts with the results for the optimal value given in previous papers (\citet{AdaptiveRobustIMRT,werner2010consistency,sun2015convergence}).

\item \textbf{Closedness and upper semi-continuity of the optimal solution set}. Under the above outlined assumptions, our results relax restrictions such as the type of uncertainty set considered (\citet{AdaptiveRobustIMRT,moazeni2013regularized}) and the specific form or limit of its convergence (\citet{AdaptiveRobustIMRT,werner2010consistency}) that are found in previous papers. Furthermore, our results hold for both uncertainty in the constraint set and cost function, as compared with the results in previous papers (\citet{moazeni2013regularized,sun2015convergence}) which only include uncertainty in the cost function. 



\item \textbf{Lipschitz continuity of the $\epsilon$-optimal solution set}. This is the first quantitative stability result for the $\epsilon$-optimal solution set in RO and robust LSIO, for which not even qualitative stability results have been established yet. Our proof leverages a similar result from variational analysis (\citet{rockafellar2009variational}).
\end{enumerate}

The rest of the paper is organized as follows. In Section~\ref{sec:Prelim}, we review relevant background and introduce our notation. The main stability results we prove require two things: first, a mapping between RO and LSIO problems, which we develop in Section~\ref{sec:RO_LSIO_Transform}, and second, a guarantee of invariance of the Lipschitz constant for equivalent LSIO problems, which we prove in Section~\ref{sec:Lipschitz_constant_invariance}.  Finally, we prove the desired RO and robust LSIO stability results of Lipschitz continuity of the optimal value in Section~\ref{sec:Lipschitz_continuity_optimal_value}, set convergence of the optimal solution set in Section~\ref{sec:Convergence_results_optimal_set}, and Lipschitz continuity of the $\epsilon$-optimal solution set in Section~\ref{sec:approx_optimal_set}. The Appendix contains helpful technical results. While some of these results have been used implicitly in other papers such as \citet{canovas2006lipschitz}, we include them here with proofs for ease of reference.




\section{Preliminaries}\label{sec:Prelim}


\subsection{Notation}

We write $\interior(A)$, $\bd(A)$, $\cl(A)$, and $\ext(A)$ to denote the interior, boundary, closure, and exterior of a set $A$. The appropriate topology or metric space for these definitions will be clear in context. We define $\mbox{conv}(A)$ and $\mbox{cone}(A)$ as the convex hull and conical hull
of $A$, respectively. Given a metric $d$ on a metric space $X$, we define the distance between a point $x$ and a set $A\subseteq X$ as $d(x,A):=\inf_{y\in A}d(x,y)$. If $X$ is additionally a normed vector space with norm $\|\cdot\|$ and $d$ is associated with the norm $\|\cdot\|$, we can also analogously denote $d_{*}$ as the dual norm of $d$. Also, let $\chi_{A}(x)$ denote the characteristic function with $\chi_{A}(x)=0$ if $x\in A$ and $\chi_{A}(x)=+\infty$ otherwise. Let $B(x_{0},\epsilon)$ denote the ball in the appropriate metric space of radius $\epsilon$ centred at $x_{0}$. Let $B:=B(0,1)$ denote the unit ball.

\subsection{Linear Semi-Infinite Optimization Problems}

In this section, we present the LSIO concepts that are most relevant to our development. Much of the concepts and notation are adapted from \citet{canovas2006lipschitz} and related papers (\citet{canovas2005distance,canovas2006ill,canovas1999stability,lopez2012stability}).

\subsubsection{Formulation and Related Quantities}
A Linear Semi-Infinite Optimization problem takes the form:
\begin{equation}\label{eq:LSIO}
\begin{alignedat}{3}
\underset{x\in \mathbb{R}^{n}}{\inf} & \quad \langle c, x\rangle\\
\mathrm{subject \hspace{1mm} to}        & \quad \langle a_{t}, x\rangle \geq b_{t}, &&\quad  \forall t\in T,
\end{alignedat}
\end{equation}
%
where the $\langle \cdot , \cdot \rangle$ denotes the dot product in $\mathbb{R}^{n}$. Here, $c$ is the cost function, $a_{t}$ are the coefficient vectors indexed by $t$, and $b_{t}$ are the right-hand side values. The set $T$ is possibly uncountably infinite, with no assumed topological structure. It is helpful to think of the coefficients and right-hand sides as functions, namely, $a : T\to \mathbb{R}^{n}$ with $a : t \mapsto a_{t}$ and $b: T\to \mathbb{R}$ with $b:t \mapsto b_{t}$.

We denote $\pi:=(c,\sigma)$ as the tuple that uniquely defines a LSIO problem of the form \eqref{eq:LSIO}. Further, we write $\sigma:=(a,b)$ as the constraint system of $\pi$, that is, $\pi:=(c,(a,b))=(c,\sigma)$. Different LSIO problems will be indexed by $j$ and written $\pi_{j}=(c^{j},(a^{j},b^{j}))=(c^{j},\sigma_{j})$. We denote $\Pi:=\mathbb{R}^{n}\times \Sigma$ as the collection of all LSIO problems, where $\Sigma:= (\mathbb{R}^{n}\times \mathbb{R})^{T}$ is the collection of all constraint systems (both feasible and infeasible).

A problem $\pi = (c,\sigma)$ has a feasible solution set, an optimal value, an optimal solution set and, for some given $\epsilon>0$, an $\epsilon$-approximate\footnote{We will write `$\epsilon$-optimal solution set' going forward.} optimal solution set, which are denoted as:
\begin{eqnarray*}
F(\sigma) &:=& \{ x\in \mathbb{R}^{n} \mbox{ : } \langle a_{t}, x \rangle \geq b_{t}, \mbox{ } \forall t\in T\}, \\
\nu(\pi) &:=& \inf_{x\in\mathbb{R}^{n}} \{ \langle c, x \rangle \mbox{ : } x \in F(\sigma) \},\\
F^{\opt}(\pi) &:=& \{ x \in F(\sigma) \mbox{ : } \langle c , x \rangle = \nu(\pi) \},\\
F^{\epsilon\mhyphen\opt}(\pi) &:=& \{ x \in F(\sigma) \mbox{ : } \langle c , x \rangle \leq \nu(\pi) +\epsilon \}.
\end{eqnarray*}
For completeness, define $\nu(\pi)=+\infty$ if $F(\sigma)=\varnothing$.

%

Lastly, LSIO problems of the form \eqref{eq:LSIO} satisfy the strong Slater condition with strong Slater constant $\rho$, if there exists $x_{0}\in\mathbb{R}^{n}$ and $\rho>0$ such that $\langle a_{t}, x_{0}\rangle \geq b_{t} + \rho$ for all $t\in T$.

\subsubsection{Distance Metrics}\label{subsub:distance_metrics}

Let $\|\cdot\|$ denote the Euclidean norm in the context-appropriate dimension. Let two constraint systems $\sigma_{1}$ and $\sigma_{2}$ be given. For a fixed $t$, let $(a_{t}^{1},b_{t}^{1})\in \mathbb{R}^{n+1}$ and $(a_{t}^{2},b_{t}^{2})\in \mathbb{R}^{n+1}$ be the $t$-th constraint of $\sigma_{1}$ and $\sigma_{2}$, respectively. Let $d((a^{1}_{t},b^{1}_{t}), (a^{2}_{t},b^{2}_{t}))$ denote the distance between these constraints:
\begin{equation*}\label{eq:small_euclidean_metric}
d((a^{1}_{t},b^{1}_{t}), (a^{2}_{t},b^{2}_{t})):=  \left\| \left(\begin{array}{c} a^{1}_{t} \\ b^{1}_{t} \end{array} \right) - \left(\begin{array}{c} a^{2}_{t} \\ b^{2}_{t} \end{array} \right) \right\|.
\end{equation*}
Let $\delta^{\Sigma}(\sigma_{1},\sigma_{2})$ denote the distance between constraint systems $\sigma_{1}$ and $\sigma_{2}$:
\begin{equation*}\label{eq:sigma_metric}
\delta^{\Sigma}(\sigma_{1},\sigma_{2}) := \sup_{t\in T} d((a^{1}_{t},b^{1}_{t}), (a^{2}_{t},b^{2}_{t})).
\end{equation*}
Also, let $\delta^{\Pi}(\pi_{1},\pi_{2})$ denote the distance between problems $\pi_{1}:=(c^{1},\sigma_{1})$ and $\pi_{2}:=(c^{2},\sigma_{2})$:
\begin{equation*}
\delta^{\Pi}(\pi_{1},\pi_{2}) := \max\{ \|c^{1}-c^{2}\|, \delta^{\Sigma}(\sigma_{1},\sigma_{2})  \}.
\end{equation*}

\subsubsection{Equivalence relations}\label{subsub:equivalence_relations}
We define equivalence relations between problems based on their constraint systems and their cost function. Let $\sigma_{1}\sim_{\Sigma}\sigma_{2}$ (``$\Sigma$-equivalence'') denote an equivalence between the constraint systems of problems $\pi_{1}=(c^{1},\sigma_{1})=(c^{1},(a^{1},b^{1}))$ and $\pi_{2}=(c^{2},\sigma_{2})=(c^{2},(a^{2},b^{2}))$ when
\begin{equation*}
\{ (a^{1}_{t},b^{1}_{t}), t\in T \} = \{ (a^{2}_{t},b^{2}_{t}), t\in T \}.
\end{equation*}

Under this equivalence, problems $\pi_{1}$ and $\pi_{2}$ are equivalent up to an ordering of the indices and redundancy in the constraints, which is a stronger condition than the feasible solution sets being the same\footnote{Suppose that $\pi_{1}\sim_{\Sigma}\pi_{2}$ and that the constraint $\langle 0 , x\rangle \geq -1$ does not exist in either $\pi_{1}$ or $\pi_{2}$. If we add the constraint $\langle 0 , x\rangle \geq -1$ to $\pi_{1}$, then it is no longer the case that $\pi_{1}\sim_{\Sigma}\pi_{2}$. However, because $\langle 0 , x\rangle \geq -1$ is a trivial constraint, we still have $F(\sigma_{1})=F(\sigma_{2})$.}, $F(\sigma_{1})=F(\sigma_{2})$.  Let $\pi_{1} \sim_{\Pi} \pi_{2}$ (``$\Pi$-equivalence'') denote an equivalence between $\pi_{1}$ and $\pi_{2}$ when $c^{1}=c^{2}$ in addition to $\sigma_{1}\sim_{\Sigma}\sigma_{2}$. Note that $\Pi$-equivalence implies that $\pi_{1}$ and $\pi_{2}$ share the same optimal value, feasible solution set, optimal solution set, and $\epsilon$-optimal solution set. In this paper, we will distinguish between ``$\Pi$-equivalent'' defined above and ``equivalent'', which we use to refer to optimization problems that have the same feasible solution set, optimal value, optimal solution set, and $\epsilon$-optimal solution set.

\subsubsection{Why indexing matters in LSIO problems}

The following example illustrates an important subtlety when using the definitions in the previous two subsections. In particular, the choice of indexing in LSIO problems affects the $\delta^{\Sigma}$ distance between them.

\begin{example}[LSIO distances]
Consider the following $\Sigma$-equivalent constraint systems $\sigma_{1} = (a^{1},b^{1})$ and $\sigma_{2}=(a^{2},b^{2})$, with $T=\{ 1,2 \}$, and $a^{1}:=(a^{1}_{1},a^{1}_{2})$ and $a^{2}:=(a^{2}_{1},a^{2}_{2})$:
\begin{eqnarray*}
a^{1}_{1} = a^{2}_{1} &=& \left( \begin{array}{c} 1 \\ 0 \end{array} \right) \\
a^{1}_{2} = a^{2}_{2} &=& \left( \begin{array}{c} 0 \\ 1 \end{array} \right) \\
b^{1}_{t} = b^{2}_{t} &=& 0 \hspace{5mm}  \forall t \in T.
\end{eqnarray*}
Then $\delta^{\Sigma}(\sigma^{1},\sigma^{2})= 0$. On the other hand, if we simply swap the indexing of two of the $a$ vectors
\begin{eqnarray*}
a^{1}_{1} = a^{2}_{2} &=& \left( \begin{array}{c} 1 \\ 0 \end{array} \right) \\
a^{1}_{2} = a^{2}_{1} &=& \left( \begin{array}{c} 0 \\ 1 \end{array} \right) \\
b^{1}_{t} = b^{2}_{t} &=& 0 \hspace{5mm}  \forall t \in T,
\end{eqnarray*}
then $\delta^{\Sigma}(\sigma^{1},\sigma^{2})= \sqrt{2}$. Simply rearranging the indices generates a different $\delta^{\Sigma}$ distance, even though these constraint systems are $\Sigma$-equivalent. \hfill$\square$
\end{example}

\subsubsection{Classifying LSIO problems}

Let $\Sigma_{f}:= \{\sigma \in \Sigma \mbox{ : } F(\sigma) \neq \varnothing \}$ denote the set of feasible constraint systems and $\Sigma_{i}:= \Sigma \backslash \Sigma_{f}$ denote the set of infeasible constraint systems. Let $\Pi_{f}:=\{ \pi \in \Pi \mbox{ : } \sigma \in \Sigma_{f}\}$ denote the set of feasible problems and $\Pi_{i}:=\Pi\backslash\Pi_{f}$ denote the set of infeasible problems. Furthermore, let $\Pi_{s}:=\{ \pi \in \Pi \mbox{ : } F^{\opt}(\pi)\neq \varnothing \}$ denote the set of solvable problems.
Finally, we define $\Sigma_{\infty}:=\{\sigma\in \Sigma \mbox{ : }\delta^{\Sigma}(\sigma,\bd(\Sigma_{f}))=+\infty\}$ and $\Pi_{\infty}:=\{\pi\in \Pi \mbox{ : }\delta^{\Pi}(\pi,\bd(\Pi_{f}))=+\infty\}$. 
Useful facts about these sets of LSIO problems are stated and proved in Appendix~\ref{app:set_LSIO_problems_relationships}.


\subsection{Robust Optimization}

In this section, we present a formulation of the Robust Linear Optimization problem. We begin with a RO problem with a single constraint and then generalize to the case of multiple constraints. Although our results hold for the general case, we present results for the single-constraint case first in Sections~\ref{sec:RO_LSIO_Transform} and~\ref{sec:Lipschitz_continuity_optimal_value}, to help build intuition.

Suppose $c$ is the cost function, $a$ is the coefficient vector, and $b$ is the right-hand side value. For now, suppose $c$ and $b$ are fixed, but $a$ is uncertain and lies in an uncertainty set $U\subseteq\mathbb{R}^{n}$, which is a non-empty, compact, and convex subset of $\mathbb{R}^{n}$. Then the single-constraint RO problem is:
\begin{equation}\label{eq:singleRobust}
\begin{alignedat}{3}
\underset{x\in \mathbb{R}^{n}}{\inf} & \quad \langle c, x\rangle\\
\mathrm{subject \hspace{1mm} to}        & \quad \langle a, x\rangle \geq b, &&\quad  \forall a\in U.
\end{alignedat}
\end{equation}

We denote $\RO(U):=(c,U,b)$ as the tuple that uniquely determines the single-constraint RO problem \eqref{eq:singleRobust}.
To avoid defining new notation, we denote its feasible solution set, optimal value, optimal solution set, and $\epsilon$-optimal solution set of RO problems in the same way as in the LSIO case: $F(\RO(U))$, $\nu(\RO(U))$, $F^{\opt}(\RO(U))$, and $F^{\epsilon\mhyphen\opt}(\RO(U))$, respectively; the context will make it clear whether we are talking about LSIO or RO problems\footnote{Note that we write $F(\RO(U))$ as the feasible solution set of a robust problem, without making a distinction between a ``constraint system'' $(U,b)$ and a RO ``problem'' $(c,U,b)$.}.

We measure the distance between two uncertainty sets $U$ and $V$ using the Hausdorff distance:
\begin{equation*}
d_{H}(U,V):=\max\left\{ \adjustlimits\sup_{u\in U}\inf_{v\in V} \|u-v\|, \adjustlimits\sup_{v\in V}\inf_{u\in U} \|u-v\|  \right\}.
\end{equation*}

Now we generalize formulation \eqref{eq:singleRobust} to the multiple-constraint case, as well as the case of uncertainty in both the cost function and right-hand side values. 
Suppose $c$ is an uncertain cost function and lies in an uncertainty set $C\subseteq\mathbb{R}^{n}$. Suppose $I$ is an index set that is possibly uncountable. For each $\alpha\in I$, suppose $(a^{\alpha},b^{\alpha})$ is an uncertain constraint that lies in an uncertainty set $U_{\alpha}\subseteq\mathbb{R}^{n+1}$. Each constraint, then, has an uncertainty set that is independent of the uncertainty sets in the other constraints. For convenience, write the \emph{constraint-wise uncertainty set} $\natmbU:=\bigotimes_{\alpha\in I}U_{\alpha}$. Then, the multiple-constraint RO problem is:
\begin{equation}\label{eq:multiple_FIRST_ROBUST}
\begin{alignedat}{3}
\adjustlimits \inf_{x\in\mathbb{R}^{n}}\sup_{c\in C}  & \quad \langle c,x\rangle,\\
\mathrm{subject \hspace{1mm} to}            & \quad \langle a^{\alpha},x \rangle \geq b^{\alpha}, &&\quad  (a^{\alpha},b^{\alpha})\in U_{\alpha}\mbox{, } \forall \alpha \in I.
\end{alignedat}
\end{equation}

Without loss of generality, we prove our stability results for constraint-wise uncertainty sets of the form $\natmbU$. This is because, for any RO problem with an arbitrary uncertainty set $\mbU$ on the constraints, it is possible to formulate an RO problem of the form \eqref{eq:multiple_FIRST_ROBUST} with an uncertainty set of the form $\natmbU:=\bigotimes_{\alpha\in I}U_{\alpha}$ generated from $\mbU$ (provided $\mbU$ generates $U_{\alpha}$ that are compact and convex). See Remark \ref{rem:Remark_Transformation} for additional details.

As in the LSIO case, a RO problem satisfies the strong Slater condition with strong Slater constant $\rho$, if there exists $x_{0}\in\mathbb{R}^{n}$ and $\rho>0$ such that $\langle a^{\alpha} , x_{0} \rangle \geq b^{\alpha} + \rho$  for all $(a^{\alpha},b^{\alpha}) \in U_{\alpha}$, $\alpha \in I$. We can assume the uncertainty is only in the constraints, characterized completely by $\natmbU$, with fixed cost function $c$:
\begin{equation}\label{eq:multipleRobustFormINITIAL}
\begin{alignedat}{3}
\underset{x\in\mathbb{R}^{n}}{\inf} \mbox{ } & \quad \langle c,x\rangle,\\
\mathrm{subject \hspace{1mm} to}            & \quad \langle a^{\alpha},x \rangle \geq b^{\alpha}, &&\quad  (a^{\alpha},b^{\alpha})\in U_{\alpha}\mbox{, } \forall \alpha \in I,
\end{alignedat}
\end{equation}
which is transformed from the formulation~\eqref{eq:multiple_FIRST_ROBUST} through an epigraph reformulation. Note that if \eqref{eq:multiple_FIRST_ROBUST} satisfies the strong Slater condition, \eqref{eq:multipleRobustFormINITIAL} does as well, with the same Slater constant.

We denote $\RO(\natmbU):=(\natmbU)$ as the (singleton) tuple that uniquely determines the multiple-constraint RO problem. When $I$ is uncountably infinite, formulation~\eqref{eq:multipleRobustFormINITIAL} is actually a robust LSIO problem. Thus, beyond simply generalizing the single-constraint problem, the results we derive for robust linear optimization using formulation~\eqref{eq:multipleRobustFormINITIAL} extend to robust LSIO problems by definition (\citet{goberna2014post}), and give insight into the stability of robust convex optimization.

The quantities $F(\RO(\natmbU))$, $\nu(\RO(\natmbU))$, $F^{\opt}(\RO(\natmbU))$, and $F^{\epsilon\mhyphen\opt}(\RO(\natmbU))$ are defined analogously to the single-constraint case. Given two constraint-wise uncertainty sets $\natmbU$ and $\natmbV$, we define a distance metric $\mb{d}_{\natural}(\natmbU,\natmbV)$ as:
\begin{equation*}
\mb{d}_{\natural}(\natmbU,\natmbV):= \sup_{\alpha\in I}d_{H}(U_{\alpha},V_{\alpha}),
\end{equation*}
defined on the metric space of constraint-wise uncertainty sets where the uncertainty set on each constraint is compact and convex.



Table \ref{tab:notationConstants} in Appendix \ref{app:Summary_Notation} summarizes the notation used in this paper.

\subsection{Important Theorems}
The main results we present in this paper parallel and leverage the following established results on stability of LSIO and general convex optimization problems.

\subsubsection{Lipschitz continuity of the optimal value}
\begin{theorem}[Theorem 4.3 from \citet{canovas2006lipschitz}]\label{thm:LipschitzContinuityLSIO}
Let $\pi_{0}=(c^{0},\sigma_{0})\in \interior(\Pi_{s})$, and let $0\leq \epsilon <\delta^{\Pi}(\pi_{0},\bd(\Pi_{s}))$. If $\pi_{1}=(c^{1},\sigma_{1})$ and $\pi_{2}=(c^{2},\sigma_{2})$ are problems in $\Pi$ satisfying $\delta^{\Pi}(\pi_{j},\pi_{0}) \leq \epsilon$ for $j=1,2$, then
\begin{equation}
| \nu(\pi_{1}) - \nu(\pi_{2}) | \leq L(\pi_{0},\epsilon) \delta^{\Pi}(\pi_{1},\pi_{2}),
\end{equation}
where
\begin{equation}
L(\pi_{0},\epsilon) := \varphi_{*}(0) \left( (\epsilon + \|c^{0}\|)\frac{\psi(\mu(\pi_{0},\epsilon))}{\delta^{\Sigma}(\sigma_{0},\Sigma_{\mathrm{i}})-\epsilon} + \mu(\pi_{0},\epsilon)\right),
\end{equation}
with functions $\varphi_{*}(\cdot)$ and $\psi(\cdot)$, and constant $\mu(\cdot,\epsilon)$ as defined in Definition \ref{def:constantDefinition}.
\end{theorem}

The above theorem establishes Lipschitz continuity of the optimal value for LSIO problems: the difference in the optimal value between two sufficiently close LSIO problems is bounded linearly by their distance in the $\delta^{\Pi}$-metric. Note that this is a \emph{local} Lipschitz continuity result because the problems under consideration, $\pi_1$ and $\pi_2$, must reside in a neighborhood of a third problem, $\pi_0$.


\subsubsection{Closedness and Upper Semi-Continuity of optimal solution set}

Note that $F(\cdot)$ can be defined as a set mapping from the set of constraint systems $\Sigma$ to the set of subsets of $\mathbb{R}^{n}$. Similarly, $\nu(\cdot)$, $F^{\opt}(\cdot)$, and $F^{\epsilon\mhyphen\opt}(\cdot)$ can be viewed as set mappings from the set of problems $\Pi$ to the set of (extended) real numbers, in the case of the optimal value, and the set of subsets of $\mathbb{R}^{n}$, in the case of the optimal and $\epsilon$-optimal solution set. 

\begin{theorem}[Theorem 3.1 (i),(ii), and (vi) from \citet{goberna1996stability}]\label{thm:lsc_SS_interior} If $\pi:=(c,\sigma)\in \Pi_{f}$, then the following are equivalent:
\begin{enumerate}
\item $F(\cdot)$ is lower semi-continuous at $\sigma$.
\item $\pi\in \interior(\Pi_{f})$.
\item $\pi$ satisfies the strong Slater condition.
\end{enumerate}
\end{theorem}

\begin{theorem}[Theorem 5.1 (i) and (ii) from \citet{canovas1999stability}]\label{thm:optSetConv} If $\pi:=(c,\sigma)\in \Pi_{s}$, then:
\begin{enumerate}
\item $F^{\opt}(\cdot)$ is closed at $\pi$ if and only if either $F(\cdot)$ is lower semi-continuous at $\pi$ or $F(\sigma)=F^{\opt}(\pi)$.
\item If $F^{\opt}(\cdot)$ is upper semi-continuous at $\pi$, then $F^{\opt}(\cdot)$ is closed at $\pi$. The converse is true if $F^{\opt}(\pi)$ is bounded.
\end{enumerate}
\end{theorem}

These two theorems are largely technical theorems that relate the different notions of qualitative stability in LSIO problems, including closedness, lower semi-continuity, and upper semi-continuity. As stated in \citet{lopez2012stability}, closedness means that every convergent sequence composed of points from the sequence of solution sets always converges to a point in the limit set. Lower (upper) semi-continuity guarantees that the solution sets do not shrink (explode) in size in the limit.

\subsubsection{Lipschitz continuity of the $\epsilon$-optimal solution set}

\begin{theorem}[Theorem 7.69 from \citet{rockafellar2009variational}]\label{thm:estimates_for_approximately_optimal_solutions} Let $f$ and $g$ be proper, lower semi-continuous, convex functions on $\mathbb{R}^{n}$ with $\argmin f$ and $\argmin g$ non-empty. Let $r_{0}$ be large enough that $\argmin f\cap r_{0} B\neq\varnothing$ and  $\argmin g\cap r_{0} B\neq\varnothing$, and also $\min f \geq -r_{0}$ and $\min g \geq -r_{0}$. Then, with $r > r_{0}$, $\epsilon>0$ and $\ov{\eta} = \wh{\mb{d}}_{r}(f,g)$,
\begin{equation}
\wh{\mb{d}}_{r}(\epsilon \mathrm{\mhyphen argmin} f, \epsilon \mathrm{\mhyphen argmin} g) \leq \ov{\eta}\left( 1 + \frac{2r}{\ov{\eta}+\epsilon/2} \right) \leq (1+4r\epsilon^{-1})\wh{\delta}^{+}_{r}(f,g),
\end{equation}
where $\wh{\mb{d}}_{r}$ is defined in Def. \ref{def:set_distance_definitions} and $\wh{\delta}^{+}_{r}$ is defined in Def. \ref{def:function_distance_definitions}.
\end{theorem}

This theorem establishes Lipschitz continuity of the $\epsilon$-optimal set of minimizers for proper, lower semi-continuous, convex functions. In turn, this result applies to LSIO problems because, as we shall see in Section \ref{sec:approx_optimal_set}, a LSIO problem can be equivalently represented as the minimization of a proper, lower semi-continuous, convex function.

\section{RO-LSIO Transformation}\label{sec:RO_LSIO_Transform}

It is well-known that RO problems can be considered as LSIO problems. However, LSIO problems are defined with respect to some index set $T$, while RO problems are defined without reference to any index set. In addition, for a given RO problem and for \emph{any} choice of $T$ of appropriate cardinality, there are an infinite number of $\Pi$-equivalent LSIO problems defined with respect to $T$ that are equivalent to the RO problem. Thus, the challenge when comparing different RO problems (for the purpose of evaluating stability or continuity with respect to perturbations in the uncertainty set) is twofold. First, we must define a common index set $T$ so that for \emph{any} two given RO problems, the LSIO problems generated with respect to $T$ that are respectively equivalent to the RO problems have a well-defined $\delta^{\Pi}$-distance.
Second, using this $T$, we must define a RO-LSIO transformation that is a local isometry between the space of RO problems and the space of LSIO problems. That is, RO problems whose uncertainty sets are close in Hausdorff distance are mapped into respectively equivalent LSIO problems that are close in the $\delta^{\Pi}$-metric.
%
%
%
%
%
%
%
The following conceptual example illustrates these challenges.
%



\begin{example}[Indexing issues in the RO-LSIO transformation]
Consider Figures~\ref{fig:robustGoodIndexing} and~\ref{fig:robustBadIndexing}, depicting two polyhedral uncertainty sets, $U$ and $V$, from two RO problems. Since the uncertainty sets are the same in the two figures, $d_H(U,V)$ is the same in both.

Now consider the LSIO problems induced by the indexing of the vertices of $U$ and $V$ as shown, which amounts to listing out the constraints in the form of a LSIO problem \eqref{eq:LSIO}. For polyhedral uncertainty sets, it suffices to index the vertices to transform an RO problem into a LSIO problem. First, note that defining $T = \{1, \ldots, 5\}$ would be sufficient to label all the vertices of $V$, but insufficient to label all the vertices of $U$, so a more appropriate index set is $T=\{1,\ldots,6\}$.

Second, note the different LSIO problems generated by different ways of indexing the vertices in Figures~\ref{fig:robustGoodIndexing} and~\ref{fig:robustBadIndexing}. In Figure~\ref{fig:robustGoodIndexing}, the distance $\delta^{\Pi}(\pi_{U},\pi_{V})$ between the two generated LSIO problems $\pi_{U}$ and $\pi_{V}$ is characterized by the distance between vertex 5 in $U$ and vertex 5 in $V$. However, if we change the indexing of the vertices of $V$, as shown in Figure~\ref{fig:robustBadIndexing}, the distance $\delta^{\Pi}(\pi_{U},\pi'_{V})$ between the two generated LSIO problems $\pi_{U}$ and $\pi'_{V}$ is characterized by the distance between vertex 3 in $U$ and vertex 3 in $V$. Clearly $\delta^{\Pi}(\pi_{U},\pi'_{V}) > \delta^{\Pi}(\pi_{U},\pi_{V})$, even though $\pi_{V} \sim_{\Pi} \pi'_{V}$. So even in the case where $|T|$ is finite, the choice of $T$ and the specific RO-LSIO indexing affects the measurement of distance using the $\delta^{\Pi}$ metric.

\begin{figure}
\centering
\subfloat[Indexing yielding LSIO problems with small $\delta^{\Pi}$-distance.]{\label{fig:robustGoodIndexing}
   \begin{tikzpicture}[scale=0.3]

\draw[draw, draw opacity=0.5] (0,0) rectangle (17,17);

\foreach \i in {1,...,6}
  \pgfmathparse{5*cos(\i*60)}
  \let\cosAngle\pgfmathresult
  \pgfmathparse{5*sin(\i*60)}
  \let\sinAngle\pgfmathresult
    \fill [opacity=0.5] (7+\cosAngle,7+\sinAngle) circle [radius=.1] coordinate (U-\i);

\filldraw [fill opacity=0.1,varRed1, thin, draw=baseRed]
  (U-1) \foreach \i in {2,...,6}{ -- (U-\i) } -- cycle;

\node [below left] at (U-1) {\tiny{1}};
\node [below left] at (U-2) {\tiny{2}};
\node [below left] at (U-3) {\tiny{3}};
\node [below left] at (U-4) {\tiny{4}};
\node [below left] at (U-5) {\tiny{5}};
\node [below left] at (U-6) {\tiny{6}};

\node [below left = 0.5 cm] at (U-4) {\tiny{$U$}};

\pgfmathparse{25}
\let\angleShift\pgfmathresult

\pgfmathparse{10*sin(\angleShift/2)*sin((180-\angleShift)/2-30)}
\let\xshift\pgfmathresult

\pgfmathparse{10*sin(\angleShift/2)*cos((180-\angleShift)/2-30)}
\let\yshift\pgfmathresult

\foreach \i in {1,...,5}
  \pgfmathparse{5*cos(\i*72+\angleShift)}
  \let\cosAngle\pgfmathresult
  \pgfmathparse{5*sin(\i*72+\angleShift)}
  \let\sinAngle\pgfmathresult
  \fill [presentGray,opacity=0.5] (7+\xshift+\cosAngle,7+\yshift+\sinAngle) circle [radius=.1] coordinate (V-\i);

\filldraw [fill opacity=0.1,varBlue1, thin, draw=baseBlue]
  (V-1) \foreach \i in {2,...,5}{ -- (V-\i) } -- cycle;

\node [above right] at (V-1) {\tiny{1}};
\node [above right] at (V-2) {\tiny{2}};
\node [below right] at (V-2) {\tiny{3}};
\node [above right] at (V-3) {\tiny{4}};
\node [above right] at (V-4) {\tiny{5}};
\node [above right] at (V-5) {\tiny{6}};

\node [above right = 0.5 cm] at (V-1) {\tiny{$V$}};

\draw[thin,presentGray,opacity=0.5] (U-1) -- (V-1);
\draw[thin,presentGray,opacity=0.5] (U-2) -- (V-2);
\draw[thin,presentGray,opacity=0.5] (U-3) -- (V-2);
\draw[thin,presentGray,opacity=0.5] (U-4) -- (V-3);
\draw[thin,presentGray,opacity=0.5] (U-5) -- (V-4);
\draw[thin,presentGray,opacity=0.5] (U-6) -- (V-5);

\end{tikzpicture}
}
\quad
\subfloat[Indexing yielding LSIO problems with large $\delta^{\Pi}$-distance.]{\label{fig:robustBadIndexing}

\begin{tikzpicture}[scale=0.3]

\draw[draw, draw opacity=0.5] (0,0) rectangle (17,17);

\foreach \i in {1,...,6}
  \pgfmathparse{5*cos(\i*60)}
  \let\cosAngle\pgfmathresult
  \pgfmathparse{5*sin(\i*60)}
  \let\sinAngle\pgfmathresult
    \fill [opacity=0.5] (7+\cosAngle,7+\sinAngle) circle [radius=.1] coordinate (U-\i);

\filldraw [fill opacity=0.1,varRed1, thin, draw=baseRed]
  (U-1) \foreach \i in {2,...,6}{ -- (U-\i) } -- cycle;

\node [below left] at (U-1) {\tiny{1}};
\node [below left] at (U-2) {\tiny{2}};
\node [below left] at (U-3) {\tiny{3}};
\node [below left] at (U-4) {\tiny{4}};
\node [below left] at (U-5) {\tiny{5}};
\node [below left] at (U-6) {\tiny{6}};

\node [below left= 0.5 cm] at (U-4) {\tiny{$U$}};

\pgfmathparse{25}
\let\angleShift\pgfmathresult

\pgfmathparse{10*sin(\angleShift/2)*sin((180-\angleShift)/2-30)}
\let\xshift\pgfmathresult

\pgfmathparse{10*sin(\angleShift/2)*cos((180-\angleShift)/2-30)}
\let\yshift\pgfmathresult

\foreach \i in {1,...,5}
  \pgfmathparse{5*cos(\i*72+\angleShift)}
  \let\cosAngle\pgfmathresult
  \pgfmathparse{5*sin(\i*72+\angleShift)}
  \let\sinAngle\pgfmathresult
  \fill [presentGray,opacity=0.5] (7+\xshift+\cosAngle,7+\yshift+\sinAngle) circle [radius=.1] coordinate (V-\i);

\filldraw [fill opacity=0.1,varBlue1, thin, draw=baseBlue]
  (V-1) \foreach \i in {2,...,5}{ -- (V-\i) } -- cycle;

\node [above right] at (V-1) {\tiny{4}};
\node [above right] at (V-2) {\tiny{6}};
\node [below right] at (V-2) {\tiny{5}};
\node [above right] at (V-3) {\tiny{1}};
\node [above right] at (V-4) {\tiny{2}};
\node [above right] at (V-5) {\tiny{3}};

\node [above right = 0.5 cm] at (V-1) {\tiny{$V$}};

\draw[thin,presentGray,opacity=0.5] (U-1) -- (V-3);
\draw[thin,presentGray,opacity=0.5] (U-2) -- (V-4);
\draw[thin,presentGray,opacity=0.5] (U-3) -- (V-5);
\draw[thin,presentGray,opacity=0.5] (U-4) -- (V-1);
\draw[thin,presentGray,opacity=0.5] (U-5) -- (V-2);
\draw[thin,presentGray,opacity=0.5] (U-6) -- (V-2);

\end{tikzpicture}
}
\caption{Different indexing for RO problems.}
\label{fig:indexing_RO_problems}
\end{figure}
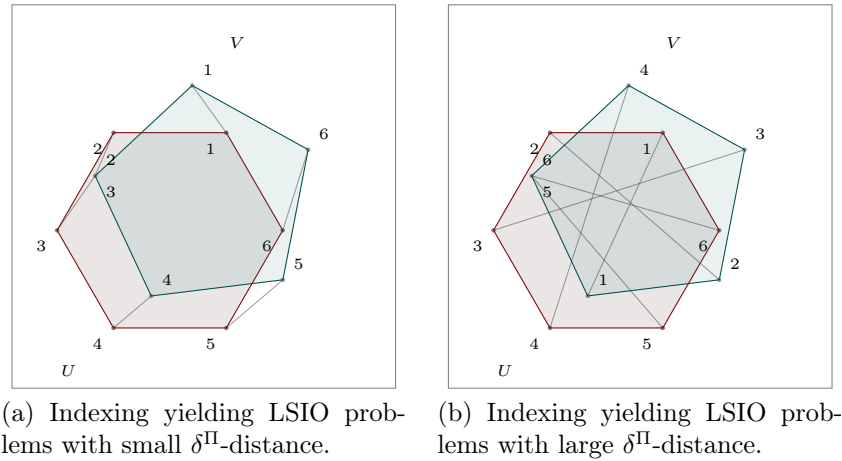




\end{example}



\subsection{Transformation theorem: single-constraint case}
To be able to index all elements from general uncertainty sets $U$ and $V$, which are uncountable subsets of $\mathbb{R}^{n}$, we choose the index set $T$ to be $\mathbb{R}^{n}$ itself. The explicit RO-LSIO transformation for $\RO(U)$ and $\RO(V)$ that we desire will yield LSIO problems whose distance in the $\delta^{\Pi}$-metric is precisely equal to $d_{H}(U,V)$. To help create such a transformation, note that repeating a constraint or adding a trivial constraint (e.g., $\langle 0_{n}, x \rangle \geq -\rho$, where $\rho>0$) can alter the distance between $\Pi$-equivalent LSIOs, but does not change the feasible solution set.



\begin{theorem}[RO-LSIO transformation for single-constraint RO problems]\label{thm:ROLSIOtransform}
Let $\rho>0$ be given. Let $\RO(U)$ and $\RO(V)$ be RO problems of the form \eqref{eq:singleRobust} with non-empty compact and convex uncertainty sets $U\subseteq\mathbb{R}^{n}$ and $V\subseteq\mathbb{R}^{n}$, cost function $c$, and right-hand side value $b$. Define $T:=\mathbb{R}^{n}$ as the common index set, with elements denoted by $t$. Define the LSIO problems $\pi_{U;V}:=(c,\sigma_{U;V})$ and $\pi_{V;U}:=(c,\sigma_{V;U})$ as follows:
\begin{equation*}
\sigma_{U;V}(t) := \begin{cases}
(t,b) & \mbox{if } t\in U, \\
(\argmin_{u\in U}\|u-t\|,b) & \mbox{if } t\in V\backslash U,\\
(0_{n},-\rho) &\mbox{if } t\notin U\cup V,
\end{cases}\\
\end{equation*}
\begin{equation*}
\sigma_{V;U}(t) := \begin{cases}
(t,b) & \mbox{if } t\in V,\\
(\argmin_{v\in V}\|v-t\|,b) & \mbox{if } t\in U\backslash V,\\
(0_{n},-\rho) &\mbox{if } t\notin U\cup V.
\end{cases}
\end{equation*}
Then $\pi_{U;V}$ and $\pi_{V;U}$ are well-defined and are equivalent to $\RO(U)$ and $\RO(V)$, respectively. Lastly, $\delta^{\Pi}(\pi_{U;V},\pi_{V;U})= d_{H}(U,V)$.
\end{theorem}

\begin{remark}
Here, we write $\argmin$ to refer to the unique element, rather than the (singleton) set, that minimizes the function in question. This abuse of notation for $\argmin$ is used in the later theorems as well.
\end{remark}

\begin{proof} \pf\ We first prove that $\pi_{U;V}$ is a well-defined LSIO problem and is equivalent to $\RO(U)$. We omit the analogous proof for $\pi_{V;U}$ and $\RO(V)$. Finally, we prove that $\delta^{\Pi}(\pi_{U;V},\pi_{V;U})= d_{H}(U,V)$.

To show that $\pi_{U;V}$ is well-defined, we must show that $\sigma_{U;V}$ maps \emph{every} element in $T:=\mathbb{R}^{n}$ to \emph{exactly one} constraint (i.e., vector in $\mathbb{R}^{n+1}$). When $t\in U$ or $t\notin U\cup V$, it is clear that $t$ maps to exactly one constraint, either $(t,b)$ or $(0_{n},-\rho)$, respectively. When $t\in V\backslash U$, the existence and uniqueness of $\argmin_{u\in U}\| u-t\|$ follows from $U$ being compact and convex, and $\|\cdot\|$ being the Euclidean norm.


To show that $\pi_{U;V}$ is equivalent to $\RO(U)$, it suffices to show that $\pi_{U;V}$ has exactly the same feasible solution set and cost function as $\RO(U)$. First, if $t\in U$, then $\sigma_{U;V}(t):=(t,b)$, so every constraint in $\RO(U)$ is listed in $\pi_{U;V}$. Second, since $\argmin_{u\in U}\| u-t\|\in U$, if $t\in V\backslash U$, then $\sigma_{U;V}(t):=(\argmin_{u\in U}\| u-t\|,b)$ is a redundant constraint. 
Lastly, if $t\notin U\cup V$, then $\sigma_{U;V}(t):=(0_{n},-\rho)$, which is a trivial constraint. Thus, every constraint in $\RO(U)$ is included in $\pi_{U;V}$ and all the other constraints of $\pi_{U;V}$ do not constrain the feasible solution set any further. Thus, the feasible solution set of $\pi_{U;V}$ is the same as $\RO(U)$, and since the cost function is the same by construction, it follows that the optimal value and optimal solution set of $\pi_{U;V}$ and $\RO(U)$ are also the same.


Lastly, we show that $\delta^{\Pi}(\pi_{U;V},\pi_{V;U}) = d_{H}(U,V)$. 
Observe that for every $t\in T$:
\begin{equation}\label{eq:four_cases_single_constraint_LSIO}
\|\sigma_{U;V}(t)-\sigma_{V;U}(t)\|=\begin{cases}
0 & \mbox{if } t\in U\cap V,\\
\min_{v\in V} \|v-t \| & \mbox{if } t \in U \backslash V, \\
\min_{u\in U} \|u-t \| & \mbox{if } t \in V \backslash U, \\
0 & \mbox{if } t \notin U\cup V,
\end{cases}
\end{equation}
where the minimum in the case $t \in U \backslash V $ follows because $\|t-\argmin_{v\in V}\|v-t\| \|= \min_{v\in V}\|v-t\|$, and analogously for the case $t \in V \backslash U$. Taking the supremum of \eqref{eq:four_cases_single_constraint_LSIO} over $t\in T$:
\begin{equation}\label{eq:actually_its_the_Hausdorff_distance}
\sup_{t\in T}\|\sigma_{U;V}(t)-\sigma_{V;U}(t)\|=
\max\left\{ \sup_{t\in U\cap V} 0 , \adjustlimits\sup_{t\in U\backslash V} \min_{v\in V} \|v-t \|, \adjustlimits\sup_{t\in V\backslash U} \min_{u\in U} \|u-t \|, \sup_{t\notin U\cup V} 0   \right\}.
\end{equation}
By definition, $\sup_{t\in T}\|\sigma_{U;V}(t)-\sigma_{V;U}(t)\| = \delta^{\Sigma}(\sigma_{U;V},\sigma_{V;U})$. Since the cost function is the same for both $\pi_{U;V}$ and $\pi_{V;U}$, it follows that $\sup_{t\in T}\|\sigma_{U;V}(t)-\sigma_{V;U}(t)\| = \delta^{\Pi}(\sigma_{U;V},\sigma_{V;U})$. Note that $\sup_{t\in U\backslash V} \min_{v\in V} \|v-t \| =  \sup_{t\in U}\min_{v\in V} \| v-t\|$ because if $t\in V$, then $\min_{v\in V}\|v-t\|=0$. Analogously, $\sup_{t\in V\backslash U} \min_{u\in U} \|u-t \| =  \sup_{t\in V}\min_{v\in U} \| u-t\|$. Finally, replacing the `$\min$' with `$\inf$' in \eqref{eq:actually_its_the_Hausdorff_distance}:
\begin{multline*}
\delta^{\Pi}(\pi_{U;V},\pi_{V;U})=\sup_{t\in T}\|\sigma_{U;V}(t)-\sigma_{V;U}(t)\|=\\
\max\left\{ \adjustlimits\sup_{t\in U} \inf_{v\in V} \|v-t \|, \adjustlimits\sup_{t\in V} \inf_{u\in U} \|u-t \|  \right\} = d_{H}(U,V).
\end{multline*}
\hfill $\square$



\end{proof}

\begin{figure}
\centering
\begin{tikzpicture}[scale=0.3]
\draw[draw, draw opacity=0.5] (0,17) rectangle (17,0);

\filldraw[varRed1, fill opacity = 0.1, thin,draw=baseRed] (9,8) circle [radius=3.5];
\node at (9,8-4) {\tiny{$U$}};

\filldraw[varBlue1, fill opacity = 0.1, thin,draw=baseBlue] (10,8.5) circle [radius=3.5];
\node at (10,8.5+4) {\tiny{$V$}};

\node[below] at (8.5,0) {$T=\mathbb{R}^{n}$};

\filldraw[presentGray, fill opacity = 0.1, thin, draw=presentGray] (9-2,8-2) circle [radius = 0.1];

\filldraw[presentGreen, fill opacity = 0.1, thin, draw=presentGreen] (9-2,8-2) circle [radius = 0.1];
\node[left] at (9-2,8-2) {\tiny{$(t,b)$}};

\node[right] at (9-2,8-2) {\tiny{$t\in U$}};

\filldraw[presentGray, fill opacity = 0.1, thin, draw=presentGray] (10+3.3,8) circle [radius = 0.1];

\filldraw[presentGreen, fill opacity = 0.1, thin, draw=presentGreen] (9+3.5,8) circle [radius = 0.1];
\node[left] at (9+3.5,8) {\tiny{$(\argmin, b)$}};

\draw[->,presentGray, opacity = 0.9] (10+3.3,8) -- (9+3.5,8);
\node[right] at (10+3.3,8) {\tiny{$t\in V\backslash U$}};

\filldraw[presentGray, fill opacity = 0.1, thin, draw=presentGray] (3,16) circle [radius = 0.1];

\filldraw[presentGreen, fill opacity = 0.1, thin, draw=presentGreen] (1,1) circle [radius = 0.1];
\node[right] at (1,1) {\tiny{$(0_{n},-\rho)$}};

\draw[->,presentGray, opacity = 0.9] (3,16) -- (1,1);
\node[right] at (3,16) {\tiny{$t\notin U\cup V$}};

\end{tikzpicture}

\caption{Defining $\pi_{U;V}$, indexing $T$. The 3 scenarios of when $t\in U$, $t\in V\backslash U$ and $t\notin U\cup V$}
\label{fig:piUV}
\end{figure}
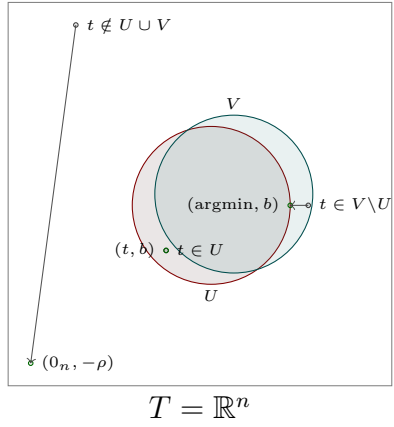

Figure \ref{fig:piUV} gives a picture of Theorem \ref{thm:ROLSIOtransform} and how we define $\pi_{U;V}$ and, analogously, $\pi_{V;U}$.\\

\subsection{Transformation theorem - multiple-constraints case}

To generalize Theorem \ref{thm:ROLSIOtransform}, we consider the case of robust problems with multiple constraints under uncertainty. We consider a special case of uncertainty for multiple constraints, as stated in the Introduction, namely, problems with fixed cost function and uncertainty sets $\natmbU$ and $\natmbV$ for the constraints. Here, $\natmbU$ and $\natmbV$  are the Cartesian products of compact and convex $U_{\alpha}$ and $V_{\alpha}$ which are subsets of $\mathbb{R}^{n}$, where $\alpha\in I$ for some constraint index set $I$. In this case, the set $T$ must be chosen as the set $I\times \mathbb{R}^{n}\times\mathbb{R}$, with elements $\mb{t}:=(\alpha,t,s)$ with $t\in \mathbb{R}^{n}$ and $s\in\mathbb{R}$. 

\begin{theorem}[RO-LSIO transformation for multiple-constraint problems]\label{thm:ROLSIOtransformMultiple}
Let $\rho>0$ be given. Let $\RO(\natmbU)$ and $\RO(\natmbV)$ be RO problems with uncertainty sets $\natmbU$ and $\natmbV$ with compact and convex $U_{\alpha}\subseteq\mathbb{R}^{n}$ and $V_{\alpha}\subseteq\mathbb{R}^{n}$ for all $\alpha\in I$ for some arbitrary index set $I$. Define $T:=I\times \mathbb{R}^{n}\times\mathbb{R}$ as the common index set, with elements denoted by $\mb{t}\in T$ which are tuples $\mb{t}:=(\alpha,t,s)$ where $\alpha\in I$, $t\in \mathbb{R}^{n}$, and $s\in\mathbb{R}$. Define the LSIO problems $\pi_{\natmbU;\natmbV}:=(c,\sigma_{\natmbU;\natmbV})$ and $\pi_{\natmbV;\natmbU}:=(c,\sigma_{\natmbV;\natmbU})$ as follows:
\begin{multline*}
\sigma_{\natmbU;\natmbV}(\mb{t}) = \sigma_{\natmbU,\natmbV}((\alpha,t,s)) := \begin{cases}
(t,s) & \mbox{if } (t,s)\in U_{\alpha},\\
\argmin_{(u_{a},u_{b})\in U_{\alpha}}d((u_{a},u_{b}),(t,s)) & \mbox{if } (t,s)\in V_{\alpha}\backslash U_{\alpha},\\
(0_{n},-\rho) &\mbox{if } (t,s)\notin U_{\alpha}\cup V_{\alpha},
\end{cases}
\end{multline*}
\begin{multline*}
\sigma_{\natmbV;\natmbU}(\mb{t})=\sigma_{\natmbV;\natmbU}((\alpha,t,s)) := \begin{cases}
(t,s) & \mbox{if } (t,s)\in V_{\alpha},\\
\argmin_{(v_{a},v_{b})\in V_{\alpha}} d((v_{a},v_{b}),(t,s)) & \mbox{if } (t,s)\in U_{\alpha}\backslash V_{\alpha},\\
(0_{n},-\rho) &\mbox{if } (t,s)\notin U_{\alpha}\cup V_{\alpha},
\end{cases}
\end{multline*}
where $d(\cdot,\cdot)$ refers to the Euclidean metric. Then $\pi_{\natmbU;\natmbV}$ and $\pi_{\natmbV;\natmbU}$ are well-defined and are equivalent to $\RO(\natmbU)$ and $\RO(\natmbV)$, respectively. Lastly, $\delta^{\Pi}(\pi_{\natmbU;\natmbV},\pi_{\natmbV;\natmbU}) = \natD(\natmbU,\natmbV):= \sup_{\alpha\in I}d_{H}(U_{\alpha},V_{\alpha})$.
\end{theorem}




\begin{proof} \pf\ We first prove that $\pi_{\natmbU;\natmbV}$ is a well-defined LSIO problem and is equivalent to $\RO(\natmbU)$. We omit the analogous proof for $\pi_{\natmbV;\natmbU}$ and $\RO(\natmbV)$. Then, we prove that $\delta^{\Pi}(\pi_{\natmbU;\natmbV},\pi_{\natmbV;\natmbU}) = \natD(\natmbU,\natmbV)$.

To show that $\pi_{\natmbU;\natmbV}$ is well-defined, we must show that $\sigma_{\natmbU,\natmbV}$ maps \emph{every} element in $T$ to \emph{exactly one} constraint (i.e., vector in $\mathbb{R}^{n+1}$). For any given $\alpha$, when $(t,s)\in U_{\alpha}$ or $(t,s)\notin U_{\alpha}\cup V_{\alpha}$, it is clear that $(t,s)$ maps to exactly one constraint, either $(t,s)$ or $(0_{n},-\rho)$. When $(t,s)\in V_{\alpha}\backslash U_{\alpha}$, the existence and uniqueness of $\argmin_{(u_{a},u_{b})\in U_{\alpha}}d((u_{a},u_{b}),(t,s))$ follows from $U_{\alpha}$ being compact and convex, and $d(\cdot,\cdot)$ being the Euclidean norm.%

To show that $\pi_{\natmbU;\natmbV}$ is equivalent to $\RO(\natmbU)$, it suffices to show that $\pi_{\natmbU;\natmbV}$ has exactly the same feasible solution set and cost function as $\RO(\natmbU)$. Let $\alpha\in I$ be chosen. First, if $(t,s)\in U_{\alpha}$, then $\sigma_{\natmbU;\natmbV}((\alpha,t,s)):=(t,s)$, thus, since we chose arbitrary $\alpha$, every constraint in $\RO(\natmbU)$ is listed in $\pi_{\natmbU;\natmbV}$. Second, since $\argmin_{(u_{a},u_{b})\in U_{\alpha}}d((u_{a},u_{b}),(t,s))\in U_{\alpha}$, if $(t,s)\in V_{\alpha}\backslash U_{\alpha}$, then $\sigma_{\natmbU;\natmbV}((\alpha,t,s)):=\argmin_{(u_{a},u_{b})\in U_{\alpha}}d((u_{a},u_{b}),(t,s))$ is a redundant constraint. Lastly, if $(t,s)\notin U_{\alpha}\cup V_{\alpha}$, then $\sigma_{\natmbU;\natmbV}((\alpha,t,s)):=(0_{n},-\rho)$, which is a trivial constraint. Thus, every constraint in $\RO(\natmbU)$ is included in $\pi_{\natmbU;\natmbV}$ and all the other constraints of $\pi_{\natmbU;\natmbV}$ do not constrain the feasible solution set any further. Thus, the feasible solution set of $\pi_{\natmbU;\natmbV}$ is the same as $\RO(\natmbU)$, and since the cost function is the same by construction, it follows that the optimal value and optimal solution set of $\pi_{\natmbU;\natmbV}$ and $\RO(\natmbU)$ are also the same.

Lastly, we show that $\delta^{\Pi}(\pi_{\natmbU;\natmbV},\pi_{\natmbV;\natmbU}) = d_{H}(\natmbU,\natmbV)$. 
Observe that for every $(\alpha,t,s)\in T$:
\begin{multline}\label{eq:four_cases_multiple_constraint_LSIO}
\|\sigma_{\natmbU;\natmbV}((\alpha,t,s))-\sigma_{\natmbV;\natmbU}((\alpha,t,s))\|=
\begin{cases}
0 & \mbox{if } (t,s)\in U_{\alpha}\cap V_{\alpha}\\
\min_{(v_{a},v_{b})\in V_{\alpha}} d((v_{a},v_{b}),(t,s)) & \mbox{if } (t,s) \in U_{\alpha} \backslash V_{\alpha} \\
\min_{(u_{a},u_{b})\in U_{\alpha}} d((u_{a},u_{b}),(t,s)) & \mbox{if } (t,s) \in V_{\alpha} \backslash U_{\alpha} \\
0 & \mbox{if } (t,s) \notin U_{\alpha}\cup V_{\alpha},
\end{cases}
\end{multline}
where the minimum in the case $t \in U_{\alpha}\backslash V_{\alpha} $ follows because $d( (t,s),\argmin_{(v_{a},v_{b})\in V_{\alpha}} d((v_{a},v_{b}),(t,s)) ) = \min_{(v_{a},v_{b})\in V_{\alpha}} d((v_{a},v_{b}),(t,s))$, by definition, and analogously for the case $t \in V_{\alpha}\backslash U_{\alpha}$. Taking the supremum of \eqref{eq:four_cases_multiple_constraint_LSIO} over $\mb{t}=(\alpha,t,s)\in T$:
\begin{multline}\label{eq:actually_its_the_Multiple_Hausdorff_distance}
\sup_{\mb{t}\in T}\|\sigma_{\natmbU;\natmbV}((\alpha,t,s))-\sigma_{\natmbV;\natmbU}((\alpha,t,s))\|=\\
\sup_{\alpha\in I}\left\{ \sup_{(t,s)\in U_{\alpha} \cap V_{\alpha}} 0 , \adjustlimits\sup_{(t,s)\in U_{\alpha}\backslash V_{\alpha} } \min_{(v_{a},v_{b}) \in V_{\alpha}}d((v_{a},v_{b}),(t,s)), \adjustlimits\sup_{(t,s)\in V_{\alpha} \backslash U_{\alpha}} \min_{(u_{a},u_{b})\in U_{\alpha}}  d((u_{a},u_{b}),(t,s)), \sup_{(t,s)\notin U_{\alpha} \cup V_{\alpha}} 0   \right\}.
\end{multline}

By definition, $\sup_{\mb{t}\in T}\|\sigma_{U;V}(\mb{t})-\sigma_{V;U}(\mb{t})\|=\delta^{\Sigma}(\sigma_{\natmbU;\natmbV},\sigma_{\natmbV;\natmbU})$. Since the cost function is the same for both $\pi_{\natmbU;\natmbV}$ and $\pi_{\natmbV;\natmbU}$, it follows that $\sup_{\mb{t}\in T}\|\sigma_{U;V}(\mb{t})-\sigma_{V;U}(\mb{t})\|=\delta^{\Pi}(\pi_{\natmbU;\natmbV},\pi_{\natmbV;\natmbU})$. 

Note that $\adjustlimits\sup_{(t,s)\in U_{\alpha}\backslash V_{\alpha}} \min_{(v_{a},v_{b}) \in V_{\alpha}} d((v_{a},v_{b}),(t,s)) = \adjustlimits \sup_{(t,s)\in U_{\alpha}}\min_{(v_{a},v_{b}) \in V_{\alpha}} d((v_{a},v_{b}),(t,s))$ because if $(t,s)\in V_{\alpha}$, then $ \min_{(v_{a},v_{b}) \in V_{\alpha}} d((v_{a},v_{b}),(t,s))=0$. 

Analogously, $\adjustlimits\sup_{(t,s)\in V_{\alpha} \backslash U_{\alpha}} \min_{(u_{a},u_{b})\in U_{\alpha}}  d((u_{a},u_{b}),(t,s)) = \adjustlimits\sup_{(t,s)\in V_{\alpha}} \min_{(u_{a},u_{b})\in U_{\alpha}}  d((u_{a},u_{b}),(t,s))$. Finally, replacing the `$\min$' with `$\inf$' in \eqref{eq:actually_its_the_Multiple_Hausdorff_distance}:
\begin{multline*}
\delta^{\Pi}(\pi_{\natmbU;\natmbV},\pi_{\natmbV;\natmbU})=\sup_{\mb{t}\in T}\|\sigma_{\natmbU;\natmbV}(\mb{t})-\sigma_{\natmbV;\natmbU}(\mb{t})\|=\\
\sup_{\alpha\in I}\left\{ \adjustlimits\sup_{(t,s)\in U_{\alpha}} \inf_{(v_{a},v_{b}) \in V_{\alpha}}d((v_{a},v_{b}),(t,s)), \adjustlimits\sup_{(t,s)\in V_{\alpha}} \inf_{(u_{a},u_{b})\in U_{\alpha}}  d((u_{a},u_{b}),(t,s)) \right\}= \sup_{\alpha\in I} d_{H}(U,V).
\end{multline*}
\hfill $\square$
\end{proof}

\section{Lipschitz constant invariance}\label{sec:Lipschitz_constant_invariance}

We first define several set mappings and present two technical lemmas that we will use in our proof of Lipschitz constant invariance.

\begin{restatable}[Set mapping definitions from \citet{canovas2006lipschitz}]{definition}{setmappingdefinition}\label{def:setMappings}
Define the following set mappings for $\pi:=(c,\sigma)=(c,(a_{t},b_{t})_{t\in T})$:
\begin{eqnarray*}
A(\sigma) &:=& \mathrm{conv}( \{ a_{t} , t\in T \} ),\\
R(\pi) &:=& \{ -b_{t} , t\in T ; \nu(\pi) \},\\
Z^{-}(\pi) &:=& \mathrm{conv}( \{ a_{t} , t\in T ; -c\} ),\\
H(\sigma) &:=& \mathrm{conv}\left( \left\{ \left( \begin{array}{c} a_{t} \\ b_{t} \end{array} \right), t\in T \right\} \right) + \left\{ \left( \begin{array}{c} 0_{n} \\ -\mu \end{array} \right), \mu \geq 0 \right\}.
\end{eqnarray*}
\end{restatable}

\begin{restatable}[Constants from \citet{canovas2006lipschitz}]{definition}{constantDefinition}\label{def:constantDefinition} For any $\pi_{0}:= (c^{0},\sigma_{0})\in \interior(\Pi_{s})$, define the following quantities for $0<\epsilon<\delta^{\Pi}(\pi_{0},\bd(\Pi_{s}))$:
\begin{eqnarray*}
\varphi(\lambda)=\varphi_{*}(\lambda) &:=& \sqrt{1+\lambda^{2}},\\
\psi(\alpha) & := & (1+\alpha)\sqrt{1+\alpha^{2}},\\
\widehat{\rho}(\pi_{0}) &:=& \frac{\sup R(\pi_{0}) }{ d_{*}(0_{n},\bd(Z^{-}(\pi_{0}))) },\\
\beta(\pi_{0},\epsilon) &:=& \frac{\psi(\widehat{\rho}(\pi_{0}))}{\delta^{\Sigma}(\sigma_{0},\Sigma_{i}) - \epsilon}, \\
\gamma(\pi_{0},\epsilon) &:=& \varphi_{*}(0)( \widehat{\rho}(\pi_{0}) + \epsilon\beta(\pi_{0},\epsilon) ) + \| c^{0} \|_{*} \beta(\pi_{0},\epsilon), \\
\mu(\pi_{0},\epsilon) &:=& \varphi(0)\frac{ \sup R(\pi_{0}) + \epsilon\max\{1,\gamma(\pi_{0},\epsilon)\}}{ d(0_{n},\bd(Z^{-}(\pi_{0}))) - \epsilon},\\
L(\pi_{0},\epsilon) &:=& \varphi_{*}(0) \left( (\epsilon + \|c^{0}\|)\frac{\psi(\mu(\pi_{0},\epsilon))}{\delta^{\Sigma}(\sigma_{0},\Sigma_{i})-\epsilon} + \mu(\pi_{0},\epsilon)\right).
\end{eqnarray*}
%
\end{restatable}
%



\begin{restatable}[Distance to ill-posedness]{lemma}{distillposed}\label{lem:distIllPosed}  If $\pi:=(c,\sigma_{0}) \in \Pi_{f}$, then
\begin{equation*}
\delta^{\Sigma}(\sigma_{0},\Sigma_{i}) = d(0_{n+1}, \bd(H(\sigma_{0}))).
\end{equation*}
\end{restatable}
\begin{proof}
\pf\ In Appendix \ref{app:proof_Lemma_dist_Ill_Posed}. \hfill $\square$
\end{proof}

\begin{restatable}[Constants in Definition \ref{def:constantDefinition} are well-defined]{lemma}{greaterThanZero}\label{lem:greaterThanZero} For any $\pi_{0}:= (c^{0},\sigma_{0})\in \interior(\Pi_{s})$, we have:
\begin{enumerate}
\item $d(0_{n+1},\bd(H(\sigma_{0})))> 0$,
\item $d(0_{n},\bd(Z^{-}(\pi_{0}))) > 0$,
\item $\delta^{\Pi}(\pi_{0},\bd(\Pi_{s}))>0$,
\item Let any $0 < \epsilon < \delta^{\Pi}(\pi_{0},\bd(\Pi_{s}))$ be given, then $\delta^{\Sigma}(\sigma_{0},\Sigma_{i})>\epsilon$. 
\end{enumerate}
In particular, all the values in Definition \ref{def:constantDefinition} are well-defined for $0< \epsilon < \delta^{\Pi}(\pi_{0},\bd(\Pi_{s}))$.
\end{restatable}
\begin{proof}
\pf\ In Appendix \ref{app:proof_Lemma_greater_than_zero}. \hfill $\square$
\end{proof}


Now we proceed to prove that given any two LSIO problems $\pi_{1},\pi_{2}\in \interior(\Pi_{s})$, such that $\pi_{1}\sim_{\Pi}\pi_{2}$, then, the Lipschitz constants in Theorem \ref{thm:LipschitzContinuityLSIO} satisfy $L(\pi_{1},\epsilon)$=$L(\pi_{2},\epsilon)$. 

\begin{lemma}[Lipschitz constant invariance]\label{lem:LipschitzInv} Let $\pi_{1}=(c^{1},\sigma_{1})\in \interior(\Pi_{s})$ and $\pi_{2}=(c^{2},\sigma_{2})\in \interior(\Pi_{s})$. Suppose $\pi_{1}\sim_{\Pi}\pi_{2}$.
Then:
\begin{enumerate}
\item $\delta^{\Pi}(\pi_{1},\bd(\Pi_{s})) = \delta^{\Pi}(\pi_{2},\bd(\Pi_{s}))$.
\item $\delta^{\Sigma}(\sigma_{1},\Sigma_{{i}}) = \delta^{\Sigma}(\sigma_{2},\Sigma_{{i}})$ and furthermore, if $0<  \epsilon < \delta^{\Pi}(\pi_{1},\bd(\Pi_{s}))=\delta^{\Pi}(\pi_{2},\bd(\Pi_{s}))$, then  $L(\pi_{1},\epsilon) = L(\pi_{2},\epsilon)$.
%
\end{enumerate}

\end{lemma}
\begin{proof} \pf\

\begin{enumerate}
\item By the definition of $\pi_{1}\sim_{\Pi} \pi_{2}$, we have $c_{1}=c_{2}$, and $\sigma_{1}:=(a^{1}_{t},b^{1}_{t})_{t\in T}\sim_{\Sigma}\sigma_{2}:=(a^{2}_{t},b^{2}_{t})_{t\in T}$, that is, $\{ (a^{1}_{t},b^{1}_{t}), t\in T \} = \{ (a^{2}_{t},b^{2}_{t}), t\in T \}$. By the definition of $H(\cdot)$, it is clear that $H(\sigma_{1})=H(\sigma_{2})$, which implies that $d(0_{n+1},\bd(H(\sigma_{1})))=d(0_{n+1},\bd(H(\sigma_{2})))$. Similarly, $Z^{-}(\pi_{1})=Z^{-}(\pi_{2})$, which implies that $d(0_{n},\bd(Z^{-}(\pi_{1})))=d(0_{n},\bd(Z^{-}(\pi_{2})))$. Since both $\pi_{1}$ and $\pi_{2}\in\interior(\Pi_{s})\subseteq \cl(\Pi_{s})$, we can apply Theorem 2 from \citet{canovas2006distance}. This theorem states that, for $\pi:=(c,\sigma)\in \cl(\Pi_{s})$,
\begin{equation*}
\delta^{\Pi}(\pi,\bd(\Pi_{s})) = \min \left\{d(0_{n+1},\bd(H(\sigma))),d(0_{n},\bd(Z^{-}(\pi)))  \right\},
\end{equation*}
which implies that $\delta^{\Pi}(\pi_{1},\bd(\Pi_{s})) = \delta^{\Pi}(\pi_{2},\bd(\Pi_{s}))$.

\item As shown in the proof of part 1 of this Lemma, $d(0_{n+1},\bd(H(\sigma_{1})))=d(0_{n+1},\bd(H(\sigma_{2})))$, so by Lemma \ref{lem:distIllPosed}, $\delta^{\Sigma}(\sigma_{1},\Sigma_{i}) = \delta^{\Sigma}(\sigma_{2},\Sigma_{i})$. Finally, $L(\pi_1,\epsilon) = L(\pi_2,\epsilon)$ for $\Pi$-equivalent problems $\pi_1, \pi_2$ follows directly from the definition of $L(\cdot, \epsilon)$ and all the supporting quantities in Definitions~\ref{def:setMappings} and~\ref{def:constantDefinition}. \hfill $\square$
%

%
%
\end{enumerate}
\end{proof}


\section{Lipschitz continuity of the optimal value}\label{sec:Lipschitz_continuity_optimal_value}

We now state and prove Lipschitz continuity for the optimal value in both the single-constraint and multiple-constraint case.
\subsection{Single-constraint}
\begin{theorem}[Lipschitz continuity for single-constraint RO] \label{thm:LipschitzSingleRO}
Let $\RO(U)$ be a RO problem of the form \eqref{eq:singleRobust} with non-empty compact and convex uncertainty set $U\subseteq\mathbb{R}^{n}$, fixed cost function $c$, and right-hand side value $b$. Suppose that:
\begin{enumerate}
\item $\RO(U)$ satisfies the strong Slater condition, with strong Slater constant $\rho>0$,
\item $F^{\opt}(\RO(U))$ is non-empty and bounded.
\end{enumerate}
Then, there exists a LSIO $\pi_{U}:=(c,\sigma_{U})\in \interior(\Pi_{s})$, where:
\begin{equation*}
\sigma_{U}(t) := \begin{cases}
(t,b) & \mbox{if } t\in U,\\
(0_{n},-\rho) & \mbox{if } t\notin U,
\end{cases}
\end{equation*}
such that for any $\epsilon$ satisfying $0 < \epsilon< \delta^{\Pi}(\pi_{U},\bd(\Pi_{s}))$, and for all compact, convex uncertainty sets $V\subseteq\mathbb{R}^{n}$ satisfying $d_{H}(U,V)< \epsilon$,
\begin{equation*}
| \nu(\RO(U)) - \nu(\RO(V)) | \leq L(U,\epsilon) d_{H}(U,V).
\end{equation*}
The Lipschitz constant $L(U,\epsilon):=L(\pi_{U},\epsilon)$ can be calculated via Definition \ref{def:constantDefinition} using $\pi_{U}:=(c,\sigma_{U})$ and the given $\epsilon$.
\end{theorem}


\begin{proof}
\pf\
The proof proceeds in seven steps. Figure \ref{fig:IdeaEnd} gives a conceptual illustration of the proof idea. Let $T:=\mathbb{R}^{n}$.
\begin{enumerate}
\item \emph{Write $\RO(U)$ as the LSIO problem $\pi_{U}$.} 
%
First, we show that $\pi_{U}$ and $\RO(U)$ are equivalent. Note that the cost functions for the two problems are the same, and that $\pi_{U}$ contains all the constraints of $\RO(U)$ plus some additional trivial constraints. Thus $\pi_{U}$ and $\RO(U)$ have the same feasible solution set, optimal value, and optimal solution sets. Furthermore, since $\RO(U)$ satisfies the strong Slater condition with Slater constant $\rho$, and since $\langle 0_{n},x\rangle \geq -\rho +\rho = 0$ for all $x$, it follows that $\pi_{U}$ satisfies the strong Slater condition with Slater constant $\rho$.


Second, we show that $\pi_{U}\in\interior(\Pi_{s})$. By assumption, $F^{\opt}(\RO(U))$ is non-empty and bounded, so $F^{\opt}(\pi_{U})$ is non-empty and bounded. Since $\pi_{U}$ satisfies the strong Slater condition, by Theorem \ref{thm:lsc_SS_interior} we have $\pi_{U}\in\interior(\Pi_{f})$. By Proposition 1 part (vi) from \citet{canovas2006ill}, $c\in\interior(\mbox{cone}(\{ a_{t}\mbox{ : } t\in T \}))$, which, by part (vii) of the same proposition, implies that $\pi_{U}\in\interior(\Pi_{s})$.

\item \emph{Choose $V$ in an $\epsilon$-neighborhood of $U$.}
Let $\epsilon>0$ be given satisfying $0 < \epsilon< \delta^{\Pi}(\pi_{U},\bd(\Pi_{s}))$. Such an $\epsilon$ exists because $\delta^{\Pi}(\pi_{U},\bd(\Pi_{s}))>0$ by Lemma \ref{lem:greaterThanZero}. Let $V$ be any compact and convex set in $\mathbb{R}^n$ satisfying $d_{H}(U,V)\leq \epsilon < \delta^{\Pi}(\pi_{U},\bd(\Pi_{s}))$. Such a $V$ exists; e.g., the choice $V:=U+B(0,\epsilon/2)$ satisfies $d_{H}(U,V)<\epsilon$.

\item \emph{Define $\pi_{U;V}:=(c,\sigma_{U;V})$.} Define the LSIO problem $\pi_{U;V}:=(c,\sigma_{U;V})$ as in Theorem \ref{thm:ROLSIOtransform}. By the same theorem, $\pi_{U;V}$ is well-defined. Since for each $t\in V\backslash U$, $\argmin_{u\in U}\|u-t\|$ is an element of $U$, every constraint in $\pi_{U;V}$ is in $\pi_{U}$ and vice versa. Thus, $\pi_{U;V}\sim_{\Pi}\pi_{U}$. Furthermore, since $\pi_{U;V}\sim_{\Pi}\pi_{U}$ and $\pi_{U}$ satisfies the strong Slater condition with Slater constant $\rho$, it follows that $\pi_{U;V}$ also satisfies the strong Slater condition with constant $\rho>0$.


\item  \emph{Define $\pi_{V;U}:=(c,\sigma_{V;U})$.} Define the LSIO problem $\pi_{V;U}:=(c,\sigma_{V;U})$ as in Theorem \ref{thm:ROLSIOtransform}. By the same theorem, $\pi_{V;U}$ is well-defined and is equivalent to $\RO(V)$.

\item \emph{By Theorem \ref{thm:ROLSIOtransform}, $\delta^{\Pi}(\pi_{U;V},\pi_{V;U}) = d_{H}(U,V) < \epsilon$.}

\item \emph{Apply Theorem \ref{thm:LipschitzContinuityLSIO} with $\pi_{U;V}\mapsto \pi_{0}, \pi_{1}$ and $\pi_{V;U}\mapsto \pi_{2}$.} We check that the assumptions of Theorem \ref{thm:LipschitzContinuityLSIO} are satisfied:
\begin{enumerate}
\item \emph{$\pi_{U;V}\in \interior(\Pi_{s})$.} This proof is identical to the proof that $\pi_{U}\in\interior(\Pi_{s})$, given in Step 1.

\item \emph{$\epsilon < \delta^{\Pi}(\pi_{U;V},\bd(\Pi_{s}))$.} Recall that $\pi_{U;V}$ and $\pi_{U}$ are both in $\interior(\Pi_{s})$, have non-empty bounded optimal solution sets, and are $\Pi$-equivalent to each other. Thus, we have  $\epsilon < \delta^{\Pi}(\pi_{U},\bd(\Pi_{s})) =\delta^{\Pi}(\pi_{U;V},\bd(\Pi_{s}))$, where the inequality is by assumption and the equality is by Lemma \ref{lem:LipschitzInv}.

\item \emph{$\delta^{\Pi}(\pi_{U;V},\pi_{V;U})\leq\epsilon$ and $\delta^{\Pi}(\pi_{U;V},\pi_{U;V})=0\leq\epsilon$.} The first inequality comes from Step 5, and the second inequality is trivial.
\end{enumerate}

Thus, the assumptions of Theorem \ref{thm:LipschitzContinuityLSIO} are satisfied, so:
\begin{equation}\label{eq:paper_proof_Lipschitz_pre_single_RO}
| \nu(\pi_{U;V}) - \nu(\pi_{V;U}) | \leq L(\pi_{U;V},\epsilon) \delta^{\Pi}(\pi_{U;V},\pi_{V;U}),
\end{equation}
where $L(\pi_{U;V},\epsilon)$ is as defined in Definition \ref{def:constantDefinition}.

\item \emph{Lastly, show that the Lipschitz constant is independent of the choice of $V$.} 
Applying Lemma \ref{lem:LipschitzInv} to $\Pi$-equivalent problems $\pi_{U}$ and $\pi_{U;V}$, it follows that $L(\pi_{U},\epsilon) = L(\pi_{U;V},\epsilon)$. Now define $L(U,\epsilon):= L(\pi_{U},\epsilon)$. By recognizing that $\nu(\pi_{U;V})=\nu(\RO(U))$ and $\nu(\pi_{V;U})=\nu(\RO(V))$, and using $\delta^{\Pi}(\pi_{U;V},\pi_{V;U}) = d_{H}(U,V)$, we obtain the final result:
\begin{align*}
| \nu(\mb{RO}(U)) - \nu(\mb{RO}(V)) | & = | \nu(\pi_{U;V}) - \nu(\pi_{V;U}) |,
& \text{(By definition of $\pi_{U;V}$ and $\pi_{V;U}$)}\\
& \leq L(\pi_{U;V},\epsilon) \delta^{\Pi}(\pi_{U;V},\pi_{V;U}),
& \text{(By Theorem \ref{thm:LipschitzContinuityLSIO})}\\
& = L(\pi_{U;V},\epsilon) d_{H}(U,V),
& \text{(By Theorem \ref{thm:ROLSIOtransform})} \\
& = L(U,\epsilon) d_{H}(U,V).
& \text{(Since $\pi_{U;V}\sim_{\Pi}\pi_{U}$)}
\end{align*}
\hfill $\square$
\end{enumerate}
\end{proof}

\begin{figure}
\begin{center}
\begin{tikzpicture}[scale=0.3]


\draw[draw, draw opacity=0.5] (0,32) rectangle (17,15);

\node[below] at (8.5,15) {$\Pi$};

\pgfmathparse{sqrt(5^2+6^2)}
\let\radiusarc\pgfmathresult

\pgfmathparse{atan(6/5)}
\let\angle\pgfmathresult

\pgfmathparse{0}
\let\smallradius\pgfmathresult

\pgfmathparse{-\smallradius*sin(\angle)}
\let\yshift\pgfmathresult

\pgfmathparse{\smallradius*cos(\angle)}
\let\xshift\pgfmathresult

\node at (1,31) {$\Pi_{s}$};



\draw[varRed1, fill opacity=0.1, thin, draw=baseRed, line cap = round]  (5-\xshift,26-\yshift) arc (-\angle:-\angle+40:\radiusarc-\smallradius);

\draw[varRed1, fill opacity=0.1, thin, draw=baseRed, line cap = round]  (5-\xshift,26-\yshift) arc (-\angle:-\angle-2:\radiusarc-\smallradius);

\filldraw[varRed1, fill opacity = 0.1, thin,draw=baseRed] (5,26) circle [radius=0.1];
\draw[draw=baseBlue,very thin] (5,26) circle [radius=2];

\pgfmathparse{-\radiusarc*sin(-\angle+40)}
\let\yend\pgfmathresult

\pgfmathparse{\radiusarc*cos(-\angle+40)}
\let\xend\pgfmathresult

\draw[very thin,baseRed] (19+8-3.5,16+7) -- (\xend,32-\yend);

\pgfmathparse{-\radiusarc*sin(-\angle-2)}
\let\yotherend\pgfmathresult

\pgfmathparse{\radiusarc*cos(-\angle-2)}
\let\xotherend\pgfmathresult

\draw[very thin,baseRed] (19+8-3.5,16+7) -- (\xotherend,32-\yotherend);

\draw[very thin, presentGray] (2,29) -- (5,26);
\node[above] at (2,28.6) {\tiny{$\pi_{U;V}$}};


\pgfmathparse{-\radiusarc*sin(-\angle+20)}
\let\ymid\pgfmathresult

\pgfmathparse{\radiusarc*cos(-\angle+20)}
\let\xmid\pgfmathresult


\filldraw[varRed1, fill opacity = 0.1, thin,draw=baseRed] (\xmid,32-\ymid) circle [radius=0.1];
\node[left] at (\xmid,32-\ymid+1) {\tiny{$\pi_{U}$}};

\draw[very thin, presentGray] (\xmid-0.5,32-\ymid+0.7) -- (\xmid,32-\ymid);

\pgfmathparse{32-\ymid}
\let\yneighmid\pgfmathresult
\pgfmathparse{32-\ymid+1}
\let\ylabelmid\pgfmathresult

\draw[draw=baseBlue,very thin] (\xmid,32-\ymid) circle [radius=2];




\pgfmathparse{sqrt(4^2+(32-25)^2)}
\let\radiusarc\pgfmathresult

\pgfmathparse{atan((32-25)/4)}
\let\angle\pgfmathresult

\pgfmathparse{0}
\let\smallradius\pgfmathresult

\pgfmathparse{-\smallradius*sin(\angle)}
\let\yshift\pgfmathresult

\pgfmathparse{\smallradius*cos(\angle)}
\let\xshift\pgfmathresult


\draw[varBlue1, fill opacity=0.1, thin, draw=baseBlue, line cap = round]  (4-\xshift,25-\yshift) arc (-\angle:-\angle+8:\radiusarc-\smallradius);
\draw[varBlue1, fill opacity=0.1, thin, draw=baseBlue, line cap = round]  (4-\xshift,25-\yshift) arc (-\angle:-\angle-10:\radiusarc-\smallradius);

\pgfmathparse{-\radiusarc*sin(-\angle+5)}
\let\yend\pgfmathresult

\pgfmathparse{\radiusarc*cos(-\angle+5)}
\let\xend\pgfmathresult

\pgfmathparse{-\radiusarc*sin(-\angle+8)}
\let\yendAlt\pgfmathresult

\pgfmathparse{\radiusarc*cos(-\angle+8)}
\let\xendAlt\pgfmathresult

\draw[very thin,baseBlue] (19+9-3.5,16+7.5)-- (\xendAlt,32-\yendAlt);

\filldraw[varBlue1, fill opacity = 0.1, thin,draw=baseBlue] (\xend,32-\yend) circle [radius=0.1];

\draw[very thin, presentGray] (1,26) -- (\xend,32-\yend);
\node at (1,26.5) {\tiny{$\pi_{V;U}$}};

\pgfmathparse{-\radiusarc*sin(-\angle-10)}
\let\yend\pgfmathresult

\pgfmathparse{\radiusarc*cos(-\angle-10)}
\let\xend\pgfmathresult

\draw[very thin,baseBlue] (19+9-3.5,16+7.5)-- (\xend,32-\yend);

\pgfmathparse{32-\ymid}
\let\yneighmid\pgfmathresult
\pgfmathparse{32-\ymid+1}
\let\ylabelmid\pgfmathresult

\pgfmathparse{sqrt(\xmid^2+\ymid^2)+2}
\let\Piradius\pgfmathresult
\filldraw[fill=presentGray, fill opacity = 0.1, draw opacity=0] (0,32) -- (\Piradius,32cm) arc (0:-90:\Piradius);
\draw (\Piradius,32cm) arc (0:-90:\Piradius);

\filldraw[varBlue1, fill opacity = 0.1, thin,draw=baseBlue] (4,25) circle [radius=0.1];





\draw[draw, draw opacity=0.5] (18,32) rectangle (35,15);
\node[below] at (8.5+18,15) {$\mathbb{R}^{2n+1}$};

\draw [->] (19,15) -- (19,32);
\node[left] at (19,32) {\tiny{$\mathbf{c}$}};

\draw [->] (18,16) -- (35,16);
\node[below] at (35,16) {\tiny{$(\mathbf{a},b)$}};

\filldraw[varRed1, fill opacity = 0.1, thin,draw=baseRed] (19+8,16+7) circle [radius=3.5];
\node at (19+8-4,16+7-2) {\tiny{$U$}};

\filldraw[varBlue1, fill opacity = 0.1, thin,draw=baseBlue] (19+9,16+7.5) circle [radius=3.5];
\node at (19+9+4,16+7.5+2) {\tiny{$V$}};
\end{tikzpicture}
\caption{Proof idea; the picture here also illustrates cost function uncertainty, as will be discussed in the next sub-section. The thicker red and blue lines in $\Pi$-space represent LSIO problems that are $\Pi$-equivalent. Note that all the problems that are $\Pi$-equivalent to $\pi_{U}$ have the same distance from $\bd(\Pi_{s})$}
\label{fig:IdeaEnd}
\end{center}
\end{figure}
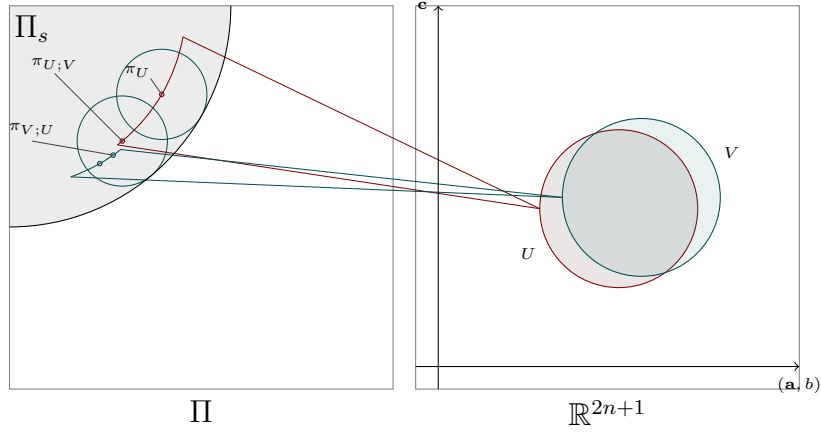


\subsection{Multiple-constraint}

\begin{restatable}[Lipschitz continuity for multiple-constraint RO]{theorem}{LipschitzMultipleRO}\label{thm:Lipschitz_multiple_coeff_cost_RHS_unc}

Let $\RO(\natmbU)$ be a RO problem of the form \eqref{eq:multipleRobustFormINITIAL} with non-empty compact and convex uncertainty set $U_{\alpha}$ in each constraint, indexed by $\alpha\in I$, with fixed cost function $c$. Suppose that:
\begin{enumerate}
\item $\RO(\natmbU)$ satisfies the strong Slater condition, with strong Slater constant $\rho>0$,
\item $F^{\opt}(\RO(\natmbU))$ is non-empty and bounded,
\item $\sup\{ -b^{\alpha}, \alpha\in I \}<+\infty$.
\end{enumerate}
Then, there exists a LSIO problem $\pi_{\natmbU}\in\interior(\Pi_{s})$, where:
\begin{equation*}
\sigma_{\natmbU}(\mb{t}) = \sigma_{\natmbU}((\alpha,t,s)) := \begin{cases}
(t, s) & \mbox{if } (t,s)\in U_{\alpha},\\
(0_{n},-\rho) & \mbox{if } (t,s)\notin U_{\alpha},
\end{cases}
\end{equation*}
such that for any $\epsilon$ satisfying $0<\epsilon < \delta^{\Pi}(\pi_{\natmbU},\bd(\Pi_{s}))$, and for all $\natmbV\subseteq\bigotimes_{\alpha\in I}(\mathbb{R}^{n}\times\mathbb{R})$ with compact and convex $V_{\alpha}$ satisfying $\sup_{\alpha\in I} d_{H}(U_{\alpha},V_{\alpha})<\epsilon$, we have:
\begin{equation*}\label{eq:multiple_dist}
|\nu(\RO(\natmbU))-\nu(\RO(\natmbV))| \leq L(\natmbU,\epsilon) \natD(\natmbU,\natmbV).
\end{equation*}
The Lipschitz constant $L(\natmbU,\epsilon):=L(\pi_{\natmbU},\epsilon)$ can be calculated via Definition \ref{def:constantDefinition} using $\pi_{\natmbU}:=(c,\sigma_{\natmbU})$ and the given $\epsilon$.

\end{restatable}

\begin{remark}
The theorem can be proved without the assumption that $\sup\{ -b^{\alpha}, \alpha\in I \}<+\infty$, but the resulting proof is much more complicated and does not yield additional insight.
\end{remark}

\begin{proof}\pf\

The proof is nearly identical to that of Theorem \ref{thm:LipschitzSingleRO}. The only differences are that $T$ is defined as $I\times \mathbb{R}^{n}\times \mathbb{R}$ and Theorem \ref{thm:ROLSIOtransformMultiple} is used in place of Theorem \ref{thm:ROLSIOtransform} where appropriate.
%
%
The full proof is given in Appendix \ref{app:Proof_Theorem_Lipschitz_Multiple}. \hfill $\square$

\end{proof}

\begin{remark}\label{rem:Remark_Transformation}
Note that Theorem \ref{thm:Lipschitz_multiple_coeff_cost_RHS_unc} can still be applied even if we are given an arbitrary uncertainty set $\mb{U}$ that is not of the form specified in~\eqref{eq:multiple_FIRST_ROBUST}. We can simply generate an uncertainty set $\natmbU$ from $\mb{U}$ as follows:
\begin{equation*}
U_{\alpha}:=\{(a^{\alpha},b^{\alpha}) \in \mathbb{R}^{n+1} \mbox{ : } \exists \mb{A}\in \mb{U} \mbox{ such that $(a^{\alpha},b^{\alpha})$ is the $\alpha$-th row of $\mb{A}$}\},
\end{equation*}
and $\natmbU:=\bigotimes_{\alpha\in I}U_{\alpha}$. We can define $\RO(\mbU)$ as follows:
\begin{equation}\label{eq:multipleRobustForm}
\begin{alignedat}{3}
\underset{x}{\inf} \mbox{ } & \quad \langle c,x\rangle,\\
\mathrm{subject \hspace{1mm} to}            & \quad \langle a^{\alpha},x\rangle \geq b^{\alpha}, &&\quad   (a^{\alpha},b^{\alpha})_{\alpha\in I}\in \mb{U}.
\end{alignedat}
\end{equation}
It is not hard to show that $\RO(\mbU)$ and $\RO(\natmbU)$ are equivalent. Thus, by applying Theorem \ref{thm:Lipschitz_multiple_coeff_cost_RHS_unc}, we can bound the optimal values of the RO problems that have uncertainty sets $\mbU$ and $\mbV$ (that have compact and convex $U_{\alpha}$ and $V_{\alpha}$) with the $\mb{d}_{\natural}$-distance between their respectively generated $\natmbU$ and $\natmbV$.

\end{remark}


\section{Closedness and upper semi-continuity of the optimal solution set} \label{sec:Convergence_results_optimal_set}

First, we provide definitions of closedness, lower semi-continuity and upper semi-continuity, as defined in \citet{lopez2012stability}. Then, we state and prove the main result of this section.

\begin{definition}[Definitions from \citet{lopez2012stability}]\label{def:definitions_Lopez} We define the following properties of set valued mappings. Suppose $\mathcal{K}$ is a set valued mapping from $\Pi$ to the power set of $\mathbb{R}^{n}$. Suppose $\mathcal{K}(\ov{\pi})\neq \varnothing$ for some given $\ov{\pi}$. Then:
\begin{enumerate}

\item$\mathcal{K}$ is \emph{closed} at $\ov{\pi}$ if for any sequence $\seq{\pi_{j}}{j=1}{\infty}\subseteq \Pi$ and $\seq{x_{j}}{j=1}{\infty}\subseteq \mathbb{R}^{n}$ such that $x_{j}\in \mathcal{K}(\pi_{j})$ for all $j$, with $\pi_{j}\to \ov{\pi}$ and $x_{j}\to \ov{x}$ for some $\ov{x}\in\mathbb{R}^{n}$, we have $\ov{x}\in \mathcal{K}(\ov{\pi})$.

\item $\mathcal{K}$ is \emph{lower semi-continuous} (lsc) at $\ov{\pi}$ if for every open set $\mathcal{O}\subseteq \mathbb{R}^{n}$ such that $\mathcal{K}(\ov{\pi})\cap \mathcal{O} \neq \varnothing$, there exists a neighborhood $\mathcal{N}$ of $\ov{\pi}$ such that for all $\pi\in \mathcal{N}$ we have $\mathcal{K}(\pi)\cap \mathcal{O}\neq\varnothing$.

\item $\mathcal{K}$ is \emph{upper semi-continuous} (usc) at $\ov{\pi}$ if for each open set $\mathcal{O}\subseteq \mathbb{R}^{n}$ such that $\mathcal{K}(\ov{\pi})\subseteq \mathcal{O}$, there exists a neighborhood $\mathcal{N}$ of $\ov{\pi}$ such that for all $\pi\in\mathcal{N}$, we have $\mathcal{K}(\pi)\subseteq \mathcal{O}$.

\end{enumerate}
\end{definition}

\begin{remark}
The definitions in Definition \ref{def:definitions_Lopez} can be extended to any set-valued mapping between a metric space and the set of subsets of a metric space. In particular, we can put the metric space of all uncertainty sets of the form $\natmbU$ with metric $\mb{d}_{\natural}$ in the place of the space of LSIO problems $\Pi$.
\end{remark}

\begin{theorem}[Closedness and upper semi-continuity of $F^{\opt}$]\label{thm:closedness_USC_optimal_set}
Let $\RO(\natmbU)$ be a RO problem of the form \eqref{eq:multipleRobustFormINITIAL} with non-empty compact and convex $U_{\alpha}$, indexed by $\alpha\in I$, and with fixed cost function $c$. Suppose that:
\begin{enumerate}
\item $\RO(\natmbU)$ satisfies the strong Slater condition, with strong Slater constant $\rho>0$,
\item $F^{\opt}(\RO(\natmbU))$ is non-empty and bounded,
\item $\sup\{ -b^{\alpha}, \alpha\in I \}<+\infty$.
\end{enumerate}
Then, with respect to the metric $\natD(\natmbU,\natmbV)$,
\begin{enumerate}
\item $F^{\opt}(\RO(\cdot))$ is closed at $\natmbU$,
\item $F^{\opt}(\RO(\cdot))$ is upper semi-continuous at $\natmbU$.
\end{enumerate}
\end{theorem}

\begin{proof}
\pf\

\begin{enumerate}
\item $F^{\opt}(\RO(\cdot))$ is closed at $\natmbU$.

Let $T:=I\times \mathbb{R}^{n} \times \mathbb{R}$ with elements $(\alpha,t,s)$ where $\alpha\in I$, $t\in\mathbb{R}^{n}$ and $s\in\mathbb{R}$. Let a sequence $\{\natmbV_{j}\}_{j=1}^{\infty} \subseteq\bigotimes_{\alpha\in I}\mathbb{R}^{n+1}$ and a sequence $\{x_{j}\}_{j=1}^{\infty}\subseteq F^{\opt}(\RO(\natmbV_{j}))$ with $\natmbV_{j} \overset{\natD}{\to} \natmbU$ and with $x_{j}\to \overline{x}$ for some $\overline{x}\in\mathbb{R}^{n}$ be given. By definition, $F^{\opt}(\RO(\cdot))$ is closed at $\natmbU$ if $\overline{x}\in F^{\opt}(\RO(\natmbU))$. The idea of the proof is to show that $x_{j}$ becomes arbitrarily close to $F(\RO(\natmbU))$ as $j\to\infty$. Then, since $\nu(\RO(\natmbV_{j}))\to\nu(\RO(\natmbU))$ as $j\to\infty$ by Theorem \ref{thm:Lipschitz_multiple_coeff_cost_RHS_unc}, it must be the case that $\overline{x}\in F^{\opt}(\RO(\natmbU))$.

\begin{enumerate}[label=\arabic*.]
\item \emph{Define a LSIO $\pi_{\natmbU}\in\interior(\Pi_{s})$ that is equivalent to $\RO(\natmbU)$.}
Let $\pi_{\natmbU}:=(c,\sigma_{\natmbU})$ be defined as in Theorem~\ref{thm:Lipschitz_multiple_coeff_cost_RHS_unc}.
%
%
Then from Theorem \ref{thm:Lipschitz_multiple_coeff_cost_RHS_unc}, $\pi_{\natmbU}\in\interior(\Pi_{s})$, is equivalent to $\RO(\natmbU)$, and satisfies the strong Slater condition . 

\item \emph{Define a LSIO $\pi_{\natmbU;\natmbV_{j}}$ that is $\Pi$-equivalent to $\pi_{\natmbU}$.}

For each $\natmbV_{j}$, label the contraint uncertainty set associated with $\alpha \in I$ as $V_{j,\alpha}$, and define $\pi_{\natmbU;\natmbV_{j}}:=(c,\sigma_{\natmbU;\natmbV_{j}})$, where:
  \begin{equation*}
\sigma_{\natmbU,\natmbV_{j}}((\alpha,t,s)) := \begin{cases}
  (t,s) & \mbox{if } (t,s)\in U_{\alpha}\\
  \argmin_{(u_{a},u_{b})\in U_{\alpha}}d((u_{a},u_{b}),(t,s)) & \mbox{if } (t,s)\in V_{j,\alpha}\backslash U_{\alpha}\\
  (0_{n},-\rho) &\mbox{if } (t,s)\notin U_{\alpha}\cup V_{j,\alpha}.
  \end{cases}
  \end{equation*}
The proof that $\pi_{\natmbU;\natmbV_{j}}$ is well-defined, is $\Pi$-equivalent to $\pi_{\natmbU}$ and that $\pi_{\natmbU;\natmbV_{j}}\in\interior(\Pi_{s})$ is given in the proof of Theorem \ref{thm:Lipschitz_multiple_coeff_cost_RHS_unc}. Furthermore, since $\pi_{\natmbU;\natmbV}\sim_{\Pi}\pi_{\natmbU}$ and $\pi_{\natmbU}$ satisfies the strong Slater condition with Slater constant $\rho$, it follows that $\pi_{\natmbU;\natmbV}$ also satisfies the strong Slater condition with constant $\rho>0$.

\item \emph{Define a LSIO $\pi_{\natmbV_{j};\natmbU}$ that is equivalent to $\RO(\natmbV)$.}

Analogously, define $\pi_{\natmbV_{j};\natmbU}:=(c,\sigma_{\natmbV_{j};\natmbU})$, where:
\begin{equation*}
\sigma_{\natmbV_{j};\natmbU}((\alpha,t,s)) := \begin{cases}
(t,s) & \mbox{if } (t,s)\in V_{j,\alpha}\\
\argmin_{(v_{a},v_{b})\in V_{j,\alpha}} d((v_{a},v_{b}),(t,s)) & \mbox{if } (t,s)\in U_{\alpha}\backslash V_{j,\alpha}\\
(0_{n},-\rho) &\mbox{if } (t,s)\notin U_{\alpha}\cup V_{j,\alpha}.
\end{cases}
\end{equation*}
The proof that this is well-defined and is equivalent to $\RO(\natmbV)$ is given in the proof of Theorem \ref{thm:Lipschitz_multiple_coeff_cost_RHS_unc}.

\item \emph{Show that $d(x_{j},F(\sigma_{\natmbU;\natmbV_{j}})) \leq \frac{\psi(\|x_{j}\|)}{\delta^{\Sigma}(\sigma_{\natmbU;\natmbV_{j}},\Sigma_{i})} \natD(\natmbU,\natmbV_{j})$ for all $j$.}

We first show that
\begin{equation}\label{eq:continuity_Inequality_Step_One}
d(x_{j},F(\sigma_{\natmbU;\natmbV_{j}})) \leq \frac{\psi(\|x_{j}\|)}{\delta^{\Sigma}(\sigma_{\natmbU;\natmbV_{j}},\Sigma_{i})} \delta^{\Sigma}(\sigma_{\natmbU;\natmbV_{j}},\sigma_{\natmbV_{j};\natmbU}).
\end{equation}
Corollary 4.1 from \citet{canovas2006lipschitz} states that for any $\sigma_{0}\in\Sigma_{f}$ and $z^{0}\in F(\sigma_{0})$,
\begin{equation*}
d(z^{0},F(\sigma)) \leq \frac{\psi(\|z^{0}\|)}{\delta^{\Sigma}(\sigma,\Sigma_{i})} \delta^{\Sigma}(\sigma,\sigma_{0}),
\end{equation*}
for all $\sigma\in\interior(\Sigma_{f})$.
Thus, inequality \eqref{eq:continuity_Inequality_Step_One} follows from the application of that corollary with $\sigma_{\natmbV_{j};\natmbU}\mapsto \sigma_{0}$ and $\sigma_{\natmbU;\natmbV_{j}}\mapsto \sigma$. What remains is to check that the assumptions of the corollary are satisfied, namely that $\sigma_{\natmbV_{j};\natmbU} \in \Sigma_{f}$ and $\sigma_{\natmbU;\natmbV_{j}} \in \interior(\Sigma_{f})$. By Step 3, $x_{j}\in F^{\opt}(\pi_{\natmbV_{j};\natmbU}) \subseteq F(\pi_{\natmbV_{j};\natmbU})$,
which implies that $\pi_{\natmbV_{j};\natmbU}$ has a non-empty feasible solution set. Thus, $\sigma_{\natmbV_{j};\natmbU} \in \Sigma_{f}$. Since $\pi_{\natmbU;\natmbV_{j}}$ satisfies the strong Slater condition, by Theorem \ref{thm:lsc_SS_interior}, $\sigma_{\natmbU;\natmbV_{j}} \in \interior(\Sigma_{f})$. 

Then, we note that
\begin{equation}\label{eq:continuity_Inequality_Step_Two}
 \delta^{\Sigma}(\sigma_{\natmbU;\natmbV_{j}},\sigma_{\natmbV_{j};\natmbU}) = \natD(\natmbU,\natmbV_{j})
\end{equation}
by Theorem \ref{thm:ROLSIOtransformMultiple}, since $\pi_{\natmbU;\natmbV_{j}}$ and $\pi_{\natmbV_{j};\natmbU}$ have the same cost function.

Combining \eqref{eq:continuity_Inequality_Step_One} and \eqref{eq:continuity_Inequality_Step_Two} yields the desired inequality:
\begin{equation}\label{eq:continuity_Inequality}
d(x_{j},F(\sigma_{\natmbU;\natmbV_{j}})) \leq \frac{\psi(\|x_{j}\|)}{\delta^{\Sigma}(\sigma_{\natmbU;\natmbV_{j}},\Sigma_{i})} \natD(\natmbU,\natmbV_{j}).
\end{equation}

\item \emph{Show that $\overline{x}\in F(\sigma_{\natmbU})$.}

Since $\pi_{\natmbU;\natmbV_{j}}$ is equivalent to $\RO(\natmbU)$, it follows that $F^{\opt}(\pi_{\natmbU;\natmbV_{j}})$ is non-empty and bounded, which implies that $\pi_{\natmbU;\natmbV_{j}}\in \interior(\Pi_{s})$ (cf. proof of Theorem~\ref{thm:Lipschitz_multiple_coeff_cost_RHS_unc}). Furthermore, since $\pi_{\natmbU;\natmbV_{j}}\sim_{\Pi}\pi_{\natmbU}$ for all $j$, then $F(\sigma_{\natmbU;\natmbV_{j}})=F(\sigma_{\natmbU})$ for all $j$ and Lemma \ref{lem:LipschitzInv} further implies that $\delta^{\Sigma}(\sigma_{\natmbU;\natmbV_{j}},\Sigma_{i})=\delta^{\Sigma}(\sigma_{\natmbU},\Sigma_{i})$. Thus, for any $j$, \eqref{eq:continuity_Inequality} can be rewritten:
\begin{equation}\label{eq:continuity_Inequality_Final}
d(x_{j},F(\sigma_{\natmbU})) \leq \frac{\psi(\|x_{j}\|)}{\delta^{\Sigma}(\sigma_{\natmbU},\Sigma_{i})} \natD(\natmbU,\natmbV_{j}).
\end{equation}
We use inequality \eqref{eq:continuity_Inequality_Final} to show that $d(x_{j},F(\sigma_{\natmbU})) \to 0$ as $j\to \infty$. This result follows because the sequence $\psi(\|x_{j}\|)$ is bounded ($\psi(\| \cdot \|)$ is continuous and $x_{j}\to \overline{x}$ implies that $\psi(\| x_{j} \|)\to \psi(\| \overline{x}\|)$) and $\natD(\natmbU,\natmbV_{j})\to 0$ as $j\to\infty$ by assumption.


Next consider a subsequence of $\left\{x_{j} \right\}_{j=1}^{\infty}$ labeled $\left\{x_{j_{k}} \right\}_{k=1}^{\infty}$ and indexed by $j_{k}$. 
For each $k$, since $d(x_{j},F(\sigma_{\natmbU})) \to 0$ as $j\to \infty$, we can choose $j_{k}$ large enough so that $d(x_{j_{k}},F(\sigma_{\natmbU}))\leq 2^{-(k+1)}$. Now, we consider a different sequence $\left\{ x'_{k} \right\}_{k=1}^{\infty} \subseteq F(\sigma_{\natmbU})$ such that $d(x_{j_{k}},x'_{k})\to 0$ as $k\to\infty$, which implies $x'_{k}\to \overline{x}$. The existence of the sequence $\{x'_{k}\}_{k=1}^{\infty}$ follows from the fact that
\begin{equation*}
d(x_{j_{k}},F(\sigma_{\natmbU}))\leq 2^{-(k+1)} \implies \inf_{x'\in F(\sigma_{\natmbU})}d(x_{j_{k}},x')\leq 2^{-(k+1)}
\end{equation*}
and for each $k$, we can define a point $x_{k}'$ in $F(\sigma_{\natmbU})$ satisfying $d(x_{j_{k}},x_{k}')<2^{-k}$.

Finally, we must show that $x'_{k}\to \ov{x}$ as $k\to\infty$. By the triangle inequality:
%
\begin{equation*}
d(x'_{k},\overline{x})\leq d(x'_{k},x_{j_{k}})+d(x_{j_{k}},\overline{x})\to 0
\end{equation*}
as $k\to\infty$. Lastly, since $F(\sigma_{\natmbU})$ is closed (it is the intersection of closed sets), and $x'_{k}\to \overline{x}$, we have that $\overline{x}\in F(\sigma_{\natmbU})$.

\item \emph{Show that $\overline{x}\in F^{\opt}(\pi_{\natmbU})=F^{\opt}(\RO(\natmbU))$.}

Now we will show that $\overline{x}$ is indeed an optimal solution of $\pi_{\natmbU;\natmbV_{j}}$ and $\RO(\natmbU)$. Theorem \ref{thm:Lipschitz_multiple_coeff_cost_RHS_unc} holds for $\RO(\natmbU)$, so that $\nu(\RO(\natmbV_{j}))\to\nu(\RO(\natmbU))$ as $k\to\infty$. Also, $x_{j}\to\overline{x}$, so that $\langle c,x_{j}\rangle \to \langle c,\overline{x}\rangle$. Thus, since  $x_{j}\in F^{\opt}(\RO(\natmbV_{j}))$, we have that $\langle c,x_{j}\rangle=\nu(\RO(\natmbV_{j}))\to \nu(\RO(\natmbU))$, which implies that $\langle c, \overline{x}\rangle = \nu(\RO(\natmbU))$, in other words, $\overline{x}\in F^{\opt}(\RO(\natmbU))$.\\

\end{enumerate}

\item $F^{\opt}(\RO(\cdot))$ is upper semi-continuous at $\natmbU$.

Let $T:=I\times \mathbb{R}^{n+1} \times \mathbb{R}^{n+1}$ with elements $(\alpha,\tau,\xi)$ where $\alpha\in I$, $\tau\in\mathbb{R}^{n+1}$ and $\xi\in\mathbb{R}^{n+1}$. Let $\mathcal{O}$ be an open set satisfying $F^{\opt}(\RO(\natmbU))\subseteq\mathcal{O}$. By definition, $F^{\opt}(\RO(\cdot))$ is upper semi-continuous at $\natmbU$ if there exists some $\epsilon>0$ such that $F^{\opt}(\natmbV)\subseteq \mathcal{O}$ for all $\natmbV$ satisfying $\natD(\natmbU,\natmbV)<\epsilon$, where $\natmbV$ is a product of compact and convex $V_{\alpha}$. The idea of the proof is to use the fact that $F^{\opt}(\RO(\cdot))$ is upper semi-continuous for LSIO problems, then use another RO-LSIO transformation to gain upper semi-continuity for RO problems. 


\begin{enumerate}[label=\arabic*.]

\item \emph{Define the LSIO problem $\pi_{\natmbU}:=(c,\sigma_{\natmbU})$.} Define $\pi_{\natmbU}:=(c,\sigma_{\natmbU})$, where:
\begin{equation*}
\sigma_{\natmbU}(\alpha,\tau,\xi) := \begin{cases}
\tau & \mbox{if } \tau\in U_{\alpha}, \xi\in \mathbb{R}^{n+1},\\
(0_{n},-\rho) & \mbox{if } \tau\notin U_{\alpha}, \xi\in \mathbb{R}^{n+1}.
\end{cases}
\end{equation*}
The proof that $\pi_{\natmbU}$ is well-defined and equivalent to $\RO(\natmbU)$ essentially follows the same proof as in Theorem \ref{thm:Lipschitz_multiple_coeff_cost_RHS_unc}.

\item \emph{$F^{\opt}$ is upper semi-continuous at $\pi_{\natmbU}$.} Since $\pi_{\natmbU}$ is equivalent to $\RO(\natmbU)$, it follows that $\pi_{\natmbU}$ satisfies the strong Slater condition and is in $\Pi_{s}$. By Theorem \ref{thm:lsc_SS_interior}, $\pi_{\natmbU}$ satisfying the strong Slater condition means that $F(\cdot)$ is lower semi-continuous at $\sigma_{\natmbU}$. By Theorem \ref{thm:optSetConv}, lower semi-continuity implies that $F^{\opt}(\cdot)$ is closed at $\pi_{\natmbU}$. Also note that $F^{\opt}(\pi_{\natmbU})$ is non-empty and bounded because $\pi_{\natmbU}$ shares the same optimal solution set as $\RO(\natmbU)$. Thus, by Theorem \ref{thm:optSetConv}, $F^{\opt}(\cdot)$ is upper semi-continuous at $\pi_{\natmbU}$.

\item \emph{With the given $\mathcal{O}$, choose $\epsilon>0$ such that  if $F^{\opt}(\pi_{\natmbU})\subseteq \mathcal{O}$, then for all $\pi$ satisfying $\delta^{\Pi}(\pi,\pi_{\natmbU})\leq \epsilon$, we have $F^{\opt}(\pi)\subseteq \mathcal{O}$.} Such an $\epsilon$ exists by definition of upper semi-continuity at $\pi_{\natmbU}$. 

\item \emph{Choose $\natmbV$ satisfying $\natD(\natmbU,\natmbV)<\epsilon$ and $\natmbV$ has compact and convex $V_{\alpha}$.}

\item \emph{Define $\pi_{\natmbV;\natmbU}$ that is equivalent to $\RO(\mbV)$.} Define $\pi_{\natmbV;\natmbU}:=(c,\sigma_{\natmbV;\natmbU})$, where:
\begin{equation*}
\sigma_{\natmbV;\natmbU}(\alpha,\tau,\xi) := \begin{cases}
\xi & \mbox{if } \tau = \argmin_{u \in U_{\alpha}} d(\xi,u), \xi\in V_{\alpha},\\
(0_{n},-\rho) & \mbox{if } \tau\notin U_{\alpha}, \xi\in\mathbb{R}^{n+1},\\
\argmin_{v\in V_{\alpha}}d(v,\tau) & \mbox{otherwise}. \\
\end{cases}
\end{equation*}
Note that this is a well-defined mapping, since both the $\argmin$ in the first and third case exist and are unique due to $U_{\alpha}$ and $V_{\alpha}$ being compact and convex. Note that every constraint in $\RO(\natmbV)$ is included in $\sigma_{\natmbV;\natmbU}$ (first case), and that all other constraints are either trivial (second case) or redundant (third case). Thus, $\pi_{\natmbV;\natmbU}$ is equivalent to $\RO(\natmbV)$.

\item \emph{$\sup_{\alpha,\tau,\xi} \|\sigma_{\natmbU}(\alpha,\tau,\xi)-\sigma_{\natmbV;\natmbU}(\alpha,\tau,\xi) \| \leq \sup_{\alpha\in I} d_{H}(U_{\alpha},V_{\alpha}) =: \natD(\natmbU,\natmbV)$.}

Rewriting $\sigma_{\natmbU}(\alpha,\tau,\xi)$:
\begin{equation*}
\sigma_{\natmbU}(\alpha,\tau,\xi) := \begin{cases}
\tau & \mbox{if } \tau = \argmin_{u \in U_{\alpha}} d(\xi,u), \xi\in V_{\alpha},\\
(0_{n},-\rho) & \mbox{if } \tau\notin U_{\alpha}, \xi\in \mathbb{R}^{n+1},\\
\tau & \mbox{otherwise}.
\end{cases}
\end{equation*}
Notice the ``otherwise'' condition occurs \emph{only if} $\tau\in U_{\alpha}$.

Then:
\begin{equation*}
\|\sigma_{\natmbU}(\alpha,\tau,\xi)-\sigma_{\natmbV;\natmbU}(\alpha,\tau,\xi) \| := \begin{cases}
\min_{u \in U_{\alpha}}d(\xi,u) & \mbox{if } \tau = \argmin_{u \in U_{\alpha}} d(\xi,u), \xi\in V_{\alpha}, \\
0 & \mbox{if } \tau\notin U_{\alpha}, \xi\in\mathbb{R}^{n+1},\\
\min_{v\in V_{\alpha}} d(v,\tau) & \mbox{otherwise}. \\
\end{cases}
\end{equation*}
Finally:
\begin{align*}
\delta^{\Pi}(\pi_{\natmbV;\natmbU},\pi_{\natmbU}) & := \sup_{\alpha,\tau,\xi} \|\sigma_{\natmbU}(\alpha,\tau,\xi)-\sigma_{\natmbV;\natmbU}(\alpha,\tau,\xi) \| \\
& \leq
\sup_{\alpha\in I} \left\{ \adjustlimits\sup_{\xi\in V_{a}, \tau\in \mathbb{R}^{n+1}} \min_{u \in U_{\alpha}}d(\xi,u) , \adjustlimits\sup_{\tau\in U_{\alpha},\xi\in \mathbb{R}^{n+1}}\min_{v\in V_{\alpha}} d(v,\tau) \right\} \\
& = \sup_{\alpha\in I} d_{H}(U_{\alpha},V_{\alpha}) \\
& =: \natD(\natmbU,\natmbV).
\end{align*}
%


Since $\delta^{\Pi}(\pi_{\natmbV;\natmbU},\pi_{\natmbU})\leq \epsilon$, we have $F^{\opt}(\pi_{\natmbV;\natmbU})\subseteq \mathcal{O}$ by step 3. Since $F^{\opt}(\pi_{\natmbV;\natmbU})=F^{\opt}(\RO(\natmbV))$, the proof is complete. \hfill $\square$\\

\end{enumerate}


\end{enumerate}

\end{proof}

\begin{remark}
Note that we used a different RO-LSIO transformation in the proof of the second part of Theorem~\ref{thm:closedness_USC_optimal_set} than the one given in Theorems~\ref{thm:ROLSIOtransform} and~\ref{thm:ROLSIOtransformMultiple}. This new RO-LSIO transformation would have also worked for the proofs of Theorem \ref{thm:LipschitzSingleRO} and \ref{thm:Lipschitz_multiple_coeff_cost_RHS_unc}, however, we would have had to use the definition $T:=I\times\mathbb{R}^{n+1}\times\mathbb{R}^{n+1}$, and lost some of the geometric intuition of those proofs.

On the other hand, we cannot use the RO-LSIO transformation from Theorems \ref{thm:ROLSIOtransform} and \ref{thm:ROLSIOtransformMultiple} in the proof of the second part of Theorem \ref{thm:closedness_USC_optimal_set}. The reason is because the definition of upper semi-continuity for LSIO problems is fairly weak in that it does not provide estimates for convergence rates. In the proofs of Theorems \ref{thm:LipschitzSingleRO} and \ref{thm:Lipschitz_multiple_coeff_cost_RHS_unc}, we exploited the fact that Theorem \ref{thm:LipschitzContinuityLSIO} had Lipschitz estimates that were invariant for $\Pi$-equivalent LSIO problems. That is, given some $\eta$ with $\pi_{\natmbU}\sim_{\Pi}\pi'_{\natmbU}$, we can choose $\epsilon$ sufficiently small so that $|\nu(\pi)-\nu(\pi_{\natmbU}) |\leq L(\pi_{\natmbU}) \epsilon < \eta$ for all $\delta^{\Pi}(\pi,\pi_{\natmbU})\leq \epsilon$, then $|\nu(\pi)-\nu(\pi'_{\natmbU}) |\leq L(\pi'_{\natmbU}) \epsilon <\eta $ for all $\delta^{\Pi}(\pi,\pi'_{\natmbU})\leq \epsilon$. In other words, the optimal value function $\nu(\cdot)$ is, in some sense, ``uniformly Lipschitz continuous'' on $\Pi$-equivalent LSIO problems.

In contrast, LSIO upper semi-continuity at $\pi_{\natmbU}$ states that, for any given open set $\mathcal{O}$, there exists a neighborhood $\delta^{\Pi}(\pi,\pi_{\natmbU})\leq \epsilon$ for which each point $\pi$ in that neighborhood satisfies $F^{\opt}(\pi)\subseteq \mathcal{O}$ if $F^{\opt}(\pi_{\natmbU})\subseteq \mathcal{O}$. Thus, given $\mathcal{O}$ and that $\pi_{\natmbU}\sim_{\Pi}\pi'_{\natmbU}$ and $F^{\opt}(\pi_{\natmbU})\subseteq \mathcal{O}$, even if $\epsilon$ guarantees $F^{\opt}(\pi)\subseteq \mathcal{O}$ for all $\pi$ satisfying $\delta^{\Pi}(\pi,\pi_{\natmbU})\leq\epsilon$, it might not be the case that  $F^{\opt}(\pi)\subseteq \mathcal{O}$ for all $\pi$ satisfying $\delta^{\Pi}(\pi,\pi'_{\natmbU})\leq\epsilon$. In other words, the optimal solution set $F^{\opt}(\cdot)$ is, in some sense, not ``uniformly continuous'', even on $\Pi$-equivalent LSIO problems.

Thus, to overcome the lack of an explicit convergence rate in in the definition of upper semi-continuity, the new RO-LSIO transformation transfers the machinery of Theorems \ref{thm:ROLSIOtransform} and \ref{thm:ROLSIOtransformMultiple} into the second $\mathbb{R}^{n+1}$ in the index set $T$.

\end{remark}




\section{Lipschitz continuity of the $\epsilon$-optimal solution set} \label{sec:approx_optimal_set}

\subsection{Set and function distance notation}

The notation in this subsection is taken from \citet{attouch1993quantitative} and \citet{rockafellar2009variational}.

For any given $r>0$, define the set $(A)_{r}:=(A\cap rB)$; recall that $B$ is the unit ball. For two sets $C$ and $D$, subsets of Euclidean space, define the ``excess of $C$ on $D$'' as $e(C,D):= \sup_{x\in C}\inf_{y\in D}d(x,y)$, where $d(x,y)$ is the Euclidean metric. Note that $d_{H}(C,D)=\max\{e(C,D),e(D,C) \}$. We define the following metrics, which are extensions to the Hausdorff distance, as in \citet{rockafellar2009variational}.
\begin{definition}[Set distance definitions]\label{def:set_distance_definitions}
Let $C$ and $D$ be two non-empty sets and $r > 0$. Define the following metrics $\mb{d}_{r}(C,D)$ and $\wh{\mb{d}}_{r}(C,D)$:
\begin{eqnarray}
\mb{d}_{r}(C,D) &:=& \max_{\|x\|\leq r}\left| d(x,C) - d(x,D) \right|,\\
\wh{\mb{d}}_{r}(C,D) &:=& \inf\left\{  \eta\geq 0 \left| \begin{aligned} C\cap r B &\subseteq D+\eta B \\ D\cap r B &\subseteq C+\eta B \end{aligned} \right. \right\}= \max\{ e((C)_{r},D), e((D)_{r},C) \}. \label{eq:first_definition_of_hat_D}
\end{eqnarray}
The two different definitions of the set distance $\wh{\mb{d}}_{r}$ given in \eqref{eq:first_definition_of_hat_D} are given in \citet{rockafellar2009variational} and \citet{attouch1993quantitative}, respectively, and can be shown to be equivalent.
\end{definition}


We also define the following metric related to the epigraph distance between functions, as defined in \citet{rockafellar2009variational}.
\begin{definition}[Function distance definitions]\label{def:function_distance_definitions}
Let $f_{1}$ and $f_{2}$ be two extended real-valued functions. Define the following measures of distance:
$\delta_{r}(f_{1},f_{2})$, $\wh{\delta}_{r}(f_{1},f_{2})$ and $\wh{\delta}^{+}_{r}(f_{1},f_{2})$:
\begin{eqnarray*}
\delta_{r}(f_{1},f_{2}) &:=& \mb{d}_{r}(\mathrm{epi} f_{1}, \mathrm{epi} f_{2}),\\
\wh{\delta}_{r}(f_{1},f_{2}) &:=& \wh{\mb{d}}_{r}(\mathrm{epi} f_{1}, \mathrm{epi} f_{2}),\\
\wh{\delta}^{+}_{r}(f_{1},f_{2}) &:=& \inf\left\{\eta \geq 0 \left| \begin{aligned}
\min_{B(x,\eta)}f_{2} &\leq \max\{f_{1}(x),-r \}+\eta \\ \min_{B(x,\eta)}f_{1} & \leq \max\{f_{2}(x),-r \}+\eta  \mbox{ }
\end{aligned}\right. \forall x\in r B
\right\},
\end{eqnarray*}
where $\mathrm{epi} f := \{ (x,\alpha)\in \mathbb{R}^{n}\times\mathbb{R} \mbox{ : } \alpha\geq f(x) \}$.
\end{definition}

\begin{definition}[$\epsilon$-minimizers of a function]
Define the set of $\epsilon$-minimizers of a function $f$ as:
\begin{equation*}
\epsilon\mhyphen\mbox{argmin} f := \{ x\in \mathbb{R}^{n} \mbox{ : } f(x)\leq \inf f +\epsilon \}.
\end{equation*}
\end{definition}

\subsection{Lipschitz continuity of the $\epsilon$-optimal solution set for LSIO problems}
Before we show Lipschitz continuity of the $\epsilon$-optimal solution set for RO problems, we first establish an analogous Lipschitz continuity result for LSIO problems, which is the focus of this subsection. 
Next, we state and prove Lemmas \ref{lem:distApproxOpt} and \ref{lem:distFeasReg}, which combine different results from \citet{rockafellar2009variational} and apply them to LSIO problems. Lemma \ref{lem:distApproxOpt} relates the distance between $\epsilon$-optimal solution sets with the distance between feasible solution sets and Lemma \ref{lem:distFeasReg} relates the distance between feasible solution sets and the $\delta^{\Pi}$-distance between the LSIO problems. Combining Lemmas \ref{lem:distApproxOpt} and \ref{lem:distFeasReg} gives us the required Lipschitz continuity of the $\epsilon$-optimal solution set for LSIO problems.

\begin{lemma}[Distance between $\epsilon$-optimal solution sets via feasible solution sets]\label{lem:distApproxOpt} Suppose $\pi_{1}:=(c,\sigma_{1})$ and $\pi_{2}:=(c,\sigma_{2})$ are LSIO problems such that there exists an $r_{0}>0$ such that $r_{0} B \cap F^{\opt}(\pi_{1}) \neq \varnothing$ and $r_{0} B \cap F^{\opt}(\pi_{2}) \neq \varnothing$ and $\nu(\pi_{1})>-r_{0}$ and $\nu(\pi_{2})>-r_{0}$.
%
Define the extended real-valued functions $f_{1}(x):=\langle c,x \rangle + \chi_{F(\sigma_{1})}(x)$ and $f_{2}(x):= \langle c,x \rangle+ \chi_{F(\sigma_{2})}(x)$, where $\chi$ is the characteristic function. Then, for all $r> r_{0}$ and for all $\epsilon>0$:
\begin{equation*}
\wh{\mb{d}}_{r}(\epsilon\mathrm{\mhyphen argmin} f_{1}, \epsilon\mathrm{\mhyphen argmin} f_{2}) \leq (1+4r\epsilon^{-1})(1+\|c\|)\wh{\mb{d}}_{r}(F(\sigma_{1}),F(\sigma_{2})).
\end{equation*}
\end{lemma}

\begin{proof} \pf\ The outline of the proof is as follows. We first construct unconstrained convex functions $f_{1}$ and $f_{2}$ whose $\epsilon$-minimizers coincide with $\epsilon$-optimal solutions of $\pi_{1}$ and $\pi_{2}$. Next, we use Theorem \ref{thm:estimates_for_approximately_optimal_solutions} to obtain an inequality relating $\wh{\mb{d}}_{r}(\epsilon\mathrm{\mhyphen argmin} f_{1}, \epsilon\mathrm{\mhyphen argmin} f_{2})$ to the epigraph distance $\wh{\delta}^{+}_{r}(f_{1},f_{2})$. Finally, we use Example 7.62 and Proposition 4.37 from \citet{rockafellar2009variational} to further bound $\wh{\delta}^{+}_{r}(f_{1},f_{2})$ by $\wh{\mb{d}}_{r}(F(\sigma_{1}),F(\sigma_{2}))$.

\begin{enumerate}
\item \emph{Define $f_{1}(x):=\langle c,x \rangle+\chi_{F(\sigma_{1})}(x)$ and $f_{2}(x):=\langle c,x \rangle+\chi_{F(\sigma_{2})}(x)$.} We will show that $f_{1}$ and $f_{2}$ are proper, lower semi-continuous, and convex functions and that the minimizers (and $\epsilon$-minimizers) of $f_{1}$ and $f_{2}$ coincide with the optimal (and $\epsilon$-optimal) solutions of $\pi_{1}$ and $\pi_{2}$.

Since $F(\sigma_{1})$ and $F(\sigma_{2})$ are non-empty, there exist $x$ and $x'$ such that $f_{1}(x)<+\infty$ and $f_{2}(x')<+\infty$. Thus, $f_{1}$ and $f_{2}$ are proper functions.

By Theorem 1.6 from \citet{rockafellar2009variational}, $f_{1}$ and $f_{2}$ are lower semi-continuous if and only if $\mbox{lev}_{\leq\alpha}f_{1}$ and $\mbox{lev}_{\leq\alpha}f_{2}$ are closed for all $\alpha\in\mathbb{R}$, where
%
%
\begin{equation*}
\mbox{lev}_{\leq\alpha}f_{i}:= \{ x\in \mathbb{R}^{n} \mbox{ : } f_{i}(x)\leq \alpha \}=\{x \in F(\sigma_{i}) \mbox{ : } \langle c,x \rangle\leq \alpha \}, \mbox{ for } i = 1, 2.
\end{equation*}
%
Let $f_{01}(x):=\langle c,x \rangle$, and note that $f_{01}$ is continuous on $\mathbb{R}^{n}$, and so $f_{01}^{-1}((-\infty,\alpha])=:\mbox{lev}_{\leq\alpha} f_{01}$ is closed. Since $F(\sigma_{1})$ is closed, $\mbox{lev}_{\leq\alpha}f_{1}=F(\sigma_{1})\cap \mbox{lev}_{\leq\alpha} f_{01}$ is closed. The proof for $\mbox{lev}_{\leq\alpha}f_{2}$ is identical.

Note that the characteristic function $\chi_{A}$ is convex if $A$ is convex. Thus, $f_{1}$ and $f_{2}$ are convex because they are the sum of convex functions.

Finally, by construction, the minimizers (and $\epsilon$-minimizers) of $f_{1}$ and $f_{2}$ are exactly the optimal solutions (and $\epsilon$-optimal solutions) of $\pi_{1}$ and $\pi_{2}$. This equivalence follows because $f_{1}(x)=+\infty$ whenever $x\notin F(\sigma_{1})$ and $f_{2}(x)=+\infty$ whenever $x\notin F(\sigma_{2})$, and so minimizing $f_{1}$ and $f_{2}$ is the same as minimizing $\langle c,x \rangle$ over $F(\sigma_1)$ and $F(\sigma_2)$, respectively.

\item \emph{Use Theorem \ref{thm:estimates_for_approximately_optimal_solutions} to get:}

 \begin{equation}\label{eq:first_ineq_Lemma_71}
 \wh{\mb{d}}_{r}(\epsilon\mathrm{\mhyphen argmin} f_{1}, \epsilon\mathrm{\mhyphen argmin} f_{2}) \leq (1+4r \epsilon^{-1})\wh{\delta}^{+}_{r}(f_{1},f_{2})
\end{equation}

 for all $r > r_{0}$.

We apply Theorem \ref{thm:estimates_for_approximately_optimal_solutions} to $f_{1}$ and $f_{2}$. Theorem \ref{thm:estimates_for_approximately_optimal_solutions} requires that $f_{1}$ and $f_{2}$ are proper, lower semi-continuous, and convex functions, which was shown earlier. The other conditions in Theorem \ref{thm:estimates_for_approximately_optimal_solutions} are satisfied by our assumptions, and the fact that the minimizers of $f_{1}$ and $f_{2}$ coincide with the optimal solutions of $\pi_{1}$ and $\pi_{2}$, so that $f_{1}$ and $f_{2}$ satisfy $\argmin f_{1}\cap r_{0} B\neq \varnothing$ and $\argmin f_{2}\cap r_{0} B\neq \varnothing$ and $\min f_{1}\geq -r_{0}$ and $\min f_{2}\geq -r_{0}$.

\item \emph{Use Example 7.62 from \citet{rockafellar2009variational} to get: $\wh{\delta}^{+}_{r}(f_{1},f_{2})\leq (1+\|c\|)\wh{\mb{d}}_{r}(F(\sigma_{1}),F(\sigma_{2}))$.}

We apply the results of Example 7.62 from \citet{rockafellar2009variational} to $f_{1}$ and $f_{2}$, by setting $\langle c,x \rangle\mapsto f_{01}=f_{02}$, $F(\sigma_{1})\mapsto C_{1}$ and $F(\sigma_{2})\mapsto C_{2}$. For any $r>0$ the constant $\kappa_{i}(r)$ from Example 7.62 can be chosen uniformly to be $\| c\|$, since $f_{01}=f_{02}=\langle c,x \rangle$ so that $|\langle c,\ov{x} \rangle -\langle c,x \rangle|\leq \| c\| \|\ov{x}-x\|$ by the Cauchy-Schwarz inequality. Furthermore, $f_{01}=f_{02}$ so that $| f_{01}(x)-f_{02}(x)|=0$ for all $x$. Thus, for all $0<r<r'<+\infty$ and $\wh{\mb{d}}_{r}(F(\sigma_{1}),F(\sigma_{2}))<r'-r$:
\begin{equation*}
\wh{\delta}^{+}_{r}(f_{1},f_{2})\leq (1+\|c\|)\wh{\mb{d}}_{r}(F(\sigma_{1}),F(\sigma_{2})).
\end{equation*}

If we apply Proposition 4.37 parts (a) and (c) from \citet{rockafellar2009variational}, we obtain
\begin{equation*}
\wh{\mb{d}}_{r}(F(\sigma_{1}),F(\sigma_{2}))\leq \mb{d}_{r}(F(\sigma_{1}),F(\sigma_{2}))) \leq \max\{d(0,F(\sigma_{1})),d(0,F(\sigma_{2})) \} +r.
\end{equation*}
Thus, for \emph{any} $r>0$, by choosing $r':=\max\{d(0,F(\sigma_{1})),d(0,F(\sigma_{2})) \} +2r$, it follows that:
\begin{equation}\label{eq:second_Lemma_Inequality}
\wh{\delta}^{+}_{r}(f_{1},f_{2})\leq (1+\|c\|)\wh{\mb{d}}_{r}(F(\sigma_{1}),F(\sigma_{2})).
\end{equation}
%



\item \emph{Combine the inequalities.} Using the inequalities \eqref{eq:first_ineq_Lemma_71} and \eqref{eq:second_Lemma_Inequality} it follows that:
\begin{equation*}
\wh{\mb{d}}_{r}(\epsilon\mathrm{\mhyphen argmin} f_{1}, \epsilon\mathrm{\mhyphen argmin} f_{2}) \leq (1+4r\epsilon^{-1})(1+\|c\|)\wh{\mb{d}}_{r}(F(\sigma_{1}),F(\sigma_{2})),
\end{equation*}
for all $r>r_{0}$. \hfill $\square$\\

\end{enumerate}

\end{proof}

\begin{lemma}[Distance between feasible solution sets]\label{lem:distFeasReg}
Let $\pi_{0}:=(c,\sigma_{0})\in \interior(\Pi_{f})$ be given. Let $\eta$ satisfying $0 < \eta < \delta^{\Sigma}(\sigma_{0},\Sigma_{i})$ be given. For all $\sigma_{1},\sigma_{2}\in \Sigma_{f}$ satisfying $\delta^{\Sigma}(\sigma_{1},\sigma_{0})\leq \eta$ and $\delta^{\Sigma}(\sigma_{2},\sigma_{0})\leq \eta$, the following inequality holds for  $\pi_{1}:=(c,\sigma_{1})$ and $\pi_{2}:=(c,\sigma_{2})$ and for all $r>0$:
\begin{equation*}
\wh{\mb{d}}_{r}(F(\sigma_{1}),F(\sigma_{2})) \leq \left( \sup_{z\in r B} \frac{\psi(\|z\|)}{\delta^{\Sigma}(\sigma_{0},\Sigma_{i})-\eta} \right) \delta^{\Pi}(\pi_{1},\pi_{2}) = \left(\frac{(1+r)\sqrt{1+r^{2}}}{\delta^{\Sigma}(\sigma_{0},\Sigma_{i})-\eta}\right)\delta^{\Pi}(\pi_{1},\pi_{2}).
\end{equation*}
\end{lemma}

\begin{proof} \pf\ For brevity, write $F(\sigma_{1})$ and $F(\sigma_{2})$ as $F_{1}$ and $F_{2}$, respectively. By \eqref{eq:first_definition_of_hat_D}, $\wh{\mb{d}}_{r}(F_{1},F_{2}):= \max\left\{e((F_{1})_{r},F_{2}) , e((F_{2})_{r},F_{1})\right\}$. We will show that both $e((F_{1})_{r},F_{2})$ and $e((F_{2})_{r},F_{1})$ are bounded by $\sup_{z\in r B}\frac{\psi(\| z\|)}{\delta^{\Sigma}(\sigma_{0},\Sigma_{i})-\eta}\delta^{\Pi}(\pi_{1},\pi_{2})$. We will only prove this inequality for $e((F_{1})_{r},F_{2})$, and omit the analogous proof for $e((F_{2})_{r},F_{1})$. Let  $\eta$ satisfying $0 < \eta < \delta^{\Sigma}(\sigma_{0},\Sigma_{i})$ and $\sigma_{1},\sigma_{2}\in \Sigma_{f}$ satisfying $\delta^{\Sigma}(\sigma_{1},\sigma_{0})\leq \eta$ and $\delta^{\Sigma}(\sigma_{2},\sigma_{0})\leq \eta$ be given.

\begin{enumerate}
\item \emph{First, $\sigma_{1},\sigma_{2}\in \interior(\Sigma_{f})$.} By assumption, $\sigma_{0}\in \interior(\Sigma_{f})$. By Corollary 1 from \citet{canovas2005distance}, since $\sigma_{0}\in \Sigma_{f}$, we have $\delta^{\Sigma}(\sigma_{0},\Sigma_{i})=\delta^{\Sigma}(\sigma_{0},\bd(\Sigma_{f}))$. Suppose to the contrary that, $\sigma_{1}\in \bd(\Sigma_{f})$. Then by definition $\delta^{\Sigma}(\sigma_{0},\bd(\Sigma_{f})) \leq \delta^{\Sigma}(\sigma_{1},\sigma_{0})\leq \eta < \delta^{\Sigma}(\sigma_{0},\Sigma_{i})= \delta^{\Sigma}(\sigma_{0},\bd(\Sigma_{f}))$, which is a contradiction. The same reasoning holds for $\sigma_{2}$.
\item \emph{Let $z^{1}\in F_{1}$ be given. Then $\sup_{z^{1}\in F_{1}\cap r B} d(z^{1},F_{2})\leq \sup_{z^{1}\in F_{1}\cap r B} \frac{\psi(\| z^{1} \|)}{\delta^{\Sigma}(\sigma_{2},\Sigma_{i})}\delta^{\Sigma}(\sigma_{1},\sigma_{2})$.}

Apply Corollary 4.1 from \citet{canovas2006lipschitz} with $\sigma_{1}\mapsto \sigma_{0}$ ($\sigma_{0}$ in this instance referring to the $\sigma_{0}$ in \emph{that} Corollary) and $\sigma_{2}\mapsto\sigma$, with the required assumption that $\sigma_{1},\sigma_{2}\in \interior(\Sigma_{f})$ being satisfied. Thus, for any $z^{1}\in F_{1}$, we have
\begin{equation*}
d(z^{1},F_{2})\leq \frac{\psi(\| z^{1} \|)}{\delta^{\Sigma}(\sigma_{2},\Sigma_{i})}\delta^{\Sigma}(\sigma_{1},\sigma_{2}).
\end{equation*}
In particular, this inequality holds for any $z^{1}\in F_{1}\cap r B$:
\begin{equation}\label{eq:the_left_side_is_exactly_eF1rF2}
\sup_{z^{1}\in F_{1}\cap r B} d(z^{1},F_{2})\leq \sup_{z^{1}\in F_{1}\cap r B} \frac{\psi(\| z^{1} \|)}{\delta^{\Sigma}(\sigma_{2},\Sigma_{i})}\delta^{\Sigma}(\sigma_{1},\sigma_{2}).
\end{equation}

\item \emph{Then $e((F_{1})_{r},F_{2})\leq \sup_{z\in r B} \frac{\psi(\| z \|)}{\delta^{\Sigma}(\sigma_{2},\Sigma_{i})}\delta^{\Sigma}(\sigma_{1},\sigma_{2}).$}

The left hand side of \eqref{eq:the_left_side_is_exactly_eF1rF2} is precisely the definition of $e((F_{1})_{r},F_{2})$. The right hand side of \eqref{eq:the_left_side_is_exactly_eF1rF2} is at most the supremum over $z\in r B$:
\begin{equation*}
e((F_{1})_{r},F_{2})\leq \sup_{z\in r B} \frac{\psi(\| z \|)}{\delta^{\Sigma}(\sigma_{2},\Sigma_{i})}\delta^{\Sigma}(\sigma_{1},\sigma_{2}).
\end{equation*}

\item \emph{Finally, $e((F_{1})_{r},F_{2}) \leq \sup_{z\in r B}\frac{\psi(\| z\|)}{\delta^{\Sigma}(\sigma_{0},\Sigma_{i})-\eta}\delta^{\Pi}(\pi_{1},\pi_{2})$.}

We first show that $\delta^{\Sigma}(\sigma_{2},\Sigma_{i}) \geq \delta^{\Sigma}(\sigma_{0},\Sigma_{i})-\eta$:
\begin{eqnarray}
\delta^{\Sigma}(\sigma_{2},\Sigma_{i}) &:=& \inf_{\sigma'\in\Sigma_{i}}\delta^{\Sigma}(\sigma_{2},\sigma')= \inf_{\sigma'\in\Sigma_{i}}\delta^{\Sigma}(\sigma',\sigma_{2})\\
&\geq & \inf_{\sigma'\in\Sigma_{i}} (\delta^{\Sigma}(\sigma',\sigma_{0}) - \delta^{\Sigma}(\sigma_{0},\sigma_{2})) \label{eq:first_inequality_by_triangle}\\
&\geq & \delta^{\Sigma}(\sigma_{0},\Sigma_{i}) -\eta, \label{eq:second_inequality_by_definition}
\end{eqnarray}
where \eqref{eq:first_inequality_by_triangle} follows from the triangle inequality and \eqref{eq:second_inequality_by_definition} follows because $\delta^{\Sigma}(\sigma_{0},\Sigma_{i}):=\inf_{\sigma'\in\Sigma_{i}}\delta^{\Sigma}(\sigma',\sigma_{0})$ and $\delta^{\Sigma}(\sigma_{0},\sigma_{2})\leq \eta$. Since $\pi_{1}$ and $\pi_{2}$ have the same cost function, $\delta^{\Sigma}(\sigma_{1},\sigma_{2})=\delta^{\Pi}(\pi_{1},\pi_{2})$.

\item \emph{Then, $\wh{\mb{d}}_{r}(F(\sigma_{1}),F(\sigma_{2})) \leq \left( \sup_{z\in r B} \frac{\psi(\|z\|)}{\delta^{\Sigma}(\sigma_{0},\Sigma_{i})-\eta} \right) \delta^{\Pi}(\pi_{1},\pi_{2}) = \left(\frac{(1+r)\sqrt{1+r^{2}}}{\delta^{\Sigma}(\sigma_{0},\Sigma_{i})-\eta}\right)\delta^{\Pi}(\pi_{1},\pi_{2}).$}

The proof of $e((F_{2})_{r},F_{1}) \leq \sup_{z\in r B}\frac{\psi(\| z\|)}{\delta^{\Sigma}(\sigma_{0},\Sigma_{i})-\eta}\delta^{\Pi}(\pi_{1},\pi_{2})$ follows by interchanging the roles of $\sigma_{1}$ and $\sigma_{2}$ in the previous steps. Thus we obtain the inequality $\wh{\mb{d}}_{r}(F(\sigma_{1}),F(\sigma_{2})) \leq \left( \sup_{z\in r B} \frac{\psi(\|z\|)}{\delta^{\Sigma}(\sigma_{0},\Sigma_{i})-\eta} \right) \delta^{\Pi}(\pi_{1},\pi_{2})$. The equality follows by the definition and monotonicity of $\psi(\|z\|)$ and that $\|z\|\leq r$ for $z\in rB$. \hfill $\square$

\end{enumerate}

\end{proof}

\begin{theorem}[Lipschitz continuity of the $\epsilon$-optimal solution set for LSIO problems]\label{thm:lipschitz_continuity_of_epsilon_approximate}
Let $\pi_{0}:=(c,\sigma_{0})\in \interior(\Pi_{f})$ be given. Let $\eta$ satisfying $0 < \eta < \delta^{\Sigma}(\sigma_{0},\Sigma_{i})$ be given. Suppose $\pi_{1}:=(c,\sigma_{1})$ and $\pi_{2}:=(c,\sigma_{2})$ are LSIO problems with the following properties:
\begin{enumerate}
\item $\delta^{\Sigma}(\sigma_{1},\sigma_{0})\leq \eta$ and $\delta^{\Sigma}(\sigma_{2},\sigma_{0})\leq \eta$.
\item There exists an $r_{0}>0$ such that $r_{0} B \cap F^{\opt}(\pi_{1}) \neq \varnothing$ and $r_{0} B \cap F^{\opt}(\pi_{2}) \neq \varnothing$ and $\nu(\pi_{1})>-r_{0}$ and $\nu(\pi_{2})>-r_{0}$.
\end{enumerate}
Define the extended real valued functions $f_{1}(x):=\langle c,x \rangle + \chi_{F(\sigma_{1})}(x)$ and $f_{2}(x):= \langle c,x \rangle+ \chi_{F(\sigma_{2})}(x)$.
Then we have, for all $r> r_{0}$:
\begin{equation*}
\wh{\mb{d}}_{r}(\epsilon\mathrm{\mhyphen argmin} f_{1}, \epsilon\mathrm{\mhyphen argmin} f_{2}) \leq (1+4r\epsilon^{-1})(1+\|c\|) \left(\frac{(1+r)\sqrt{1+r^{2}}}{\delta^{\Sigma}(\sigma_{0},\Sigma_{i})-\eta}\right)\delta^{\Pi}(\pi_{1},\pi_{2}).
\end{equation*}
\end{theorem}
\begin{proof} \pf\ The result follows from the direct application of Lemma \ref{lem:distApproxOpt} and \ref{lem:distFeasReg}.\hfill $\square$
\end{proof}

Theorem~\ref{thm:lipschitz_continuity_of_epsilon_approximate}, derived by combining variational analysis and LSIO results in the literature, provide new insight into the quantitative continuity of $\epsilon$-optimal solution sets, supplementing the results in \citet{canovas2006lipschitz} for optimal solution sets.

\subsection{Lipschitz continuity of the $\epsilon$-optimal solution set for RO problems}

Using Theorem~\ref{thm:lipschitz_continuity_of_epsilon_approximate}, we can now prove Lipschitz continuity of the $\epsilon$-optimal solution set for RO problems.

\begin{theorem}[Lipschitz continuity of the $\epsilon$-optimal solution set for RO problems]\label{thm:Lipschitz_continuity_epsilon_approx_RO}
Let $\RO(\natmbU)$ be a RO problem of the form \eqref{eq:multipleRobustFormINITIAL} with non-empty compact and convex uncertainty set $U_{\alpha}$ in each constraint, indexed by $\alpha\in I$, with fixed cost function $c$. Suppose that:
\begin{enumerate}
\item $\RO(\natmbU)$ satisfies the strong Slater condition, with strong Slater constant $\rho>0$,
\item $F^{\opt}(\RO(\natmbU))$ is non-empty and bounded, 
\item $\sup\{ -b^{\alpha}, \alpha\in I \}<+\infty$.
\end{enumerate}
Define the LSIO problem $\pi_{\natmbU}\in\interior(\Pi_{s})$, where:
\begin{equation*}
\sigma_{\natmbU}(\mb{t}) = \sigma_{\natmbU}((\alpha,t,s)) := \begin{cases}
(t, s) & \mbox{if } (t,s)\in U_{\alpha},\\
(0_{n},-\rho) & \mbox{if } (t,s)\notin U_{\alpha}.
\end{cases}
\end{equation*}
Let $\eta>0$ satisfying $0<\eta < \delta^{\Pi}(\pi_{\natmbU},\bd(\Pi_{s}))$ be given. Let
$r_{0}> 0$ satisfy $r_{0}B \cap \epsilon\mathrm{\mhyphen argmin}(\RO(\natmbU))\neq \varnothing$ and $\nu(\RO(\natmbU))>-r_{0}$. Suppose $\natmbV$, with compact and convex uncertainty sets $V_{\alpha}$, satisfies the following conditions:
\begin{enumerate}
\item $d_{\natural}(\natmbU,\natmbV)<\eta$,
\item $r_{0}B\cap \epsilon\mathrm{\mhyphen argmin}(\RO(\natmbV))\neq \varnothing$,
\item $\nu(\RO(\natmbV))>-r_{0}$.
\end{enumerate}
Then, for all $r>r_{0}$:
\begin{multline*}
\wh{\mb{d}}_{r}(\epsilon\mathrm{\mhyphen argmin} (\RO(\natmbU)), \epsilon\mathrm{\mhyphen argmin} (\RO(\natmbV))) \\
\leq (1+4r\epsilon^{-1})(1+\|c\|) \left(\frac{(1+r)\sqrt{1+r^{2}}}{\delta^{\Sigma}(\sigma_{\natmbU},\Sigma_{i})-\eta}\right)\natD(\natmbU,\natmbV).
\end{multline*}

\end{theorem}
\begin{remark}
We may replace the set of conditions on $\natmbV$ with the two conditions:
\begin{enumerate}
\item $d_{\natural}(\natmbU,\natmbV)<\min\left\{ \eta, \frac{\nu(\RO(\natmbU))+r_{0}}{L(\pi_{\natmbU},\eta)} \right\}$,
\item $r_{0}B\cap \epsilon\mathrm{\mhyphen argmin}(\RO(\natmbV))\neq \varnothing$.
\end{enumerate}
The condition $d_{\natural}(\natmbU,\natmbV)<\min\left\{ \eta, \frac{\nu(\RO(\natmbU))+r_{0}}{L(\pi_{\natmbU},\eta)} \right\}$ replaces conditions 1 and 2 in Theorem \ref{thm:Lipschitz_continuity_epsilon_approx_RO} because we can apply Theorem \ref{thm:Lipschitz_multiple_coeff_cost_RHS_unc} to guarantee:
\begin{eqnarray*}
\nu(\RO(\natmbU))-\nu(\RO(\natmbV)) & < &  L(\pi_{\natmbU},\eta) d_{\natural}(\natmbU,\natmbV),\\
& \leq & \nu(\RO(\natmbU))+r_{0},
\end{eqnarray*}
\end{remark}
and hence $\nu(\natmbV) >-r_{0}.$


\begin{proof}\pf\ The steps of the proof are almost identical to that of Theorem \ref{thm:Lipschitz_multiple_coeff_cost_RHS_unc}, with the only major difference in that we have to apply Theorem \ref{thm:lipschitz_continuity_of_epsilon_approximate} in place of Theorem \ref{thm:LipschitzContinuityLSIO}. All the steps of the proof are shown in the same way as in Theorem \ref{thm:Lipschitz_multiple_coeff_cost_RHS_unc}. \hfill $\square$\\

\end{proof}

\section{Conclusion}
In this paper, we showed that robust linear optimization and robust LSIO is continuous with respect to perturbations in the uncertainty set. In particular, by constructing an explicit RO-LSIO transformation, we showed that the optimal value and $\epsilon$-optimal solution set mappings are Lipschitz continuous, which are novel quantitative stability results for robust optimization. We also showed that the optimal solution set mapping is closed and upper semi-continuous under less restrictive conditions on the uncertainty set compared with previous results in the literature. In addition, the robust LSIO stability results may provide insight to understanding stability for general robust convex optimization.


%


%
%
%

\appendix

\section{Facts about sets of LSIO problems.}\label{app:set_LSIO_problems_relationships}

\begin{remark}[Non-emptiness of constraint sets and problem sets]\label{rem:non_Empty_Constraints}
The sets $\Sigma_{f}$, $\bd(\Sigma_{f})$ and $\Sigma_{i}$ are non-empty. Furthermore, the sets $\Pi_{f}$, $\bd(\Pi_{f})$, $\Pi_{i}$, and $\Pi_{s}$ are non-empty.

\end{remark}
\begin{proof}
\pf\
Consider constraint system $\sigma_{0}$ with $(0_{n},0)$, that is, $\langle 0_{n}, x\rangle \geq 0$, as the only constraint (up to some multiplicity). Note that $\sigma_{0}$ is in both $\Sigma_{f}$ and $\bd(\Sigma_{f})$; the latter inclusion is obtained by recognizing that the constraint systems $\sigma_{\epsilon}$ with $\langle 0_{n},x \rangle \geq \epsilon$ as the only constraint satisfy $\sigma_{\epsilon}\in \Sigma_{i}$  for any $\epsilon>0$. The fact that $\sigma_{\epsilon}\in \Sigma_{i}$  for any $\epsilon>0$ also shows that $\Sigma_{i}$ is non-empty.

The same examples with the cost function $c=0_{n}$ for all of them prove that $\Pi_{f}$, $\bd(\Pi_{f})$, $\Pi_{i}$, and $\Pi_{s}$ are non-empty. \hfill $\square$
\end{proof}



\begin{remark}[Distance to $\Pi_{\infty}$]\label{rem:distance_to_Pi_infty}
The following hold:
\begin{eqnarray*}
\sigma\in\cl(\Sigma_{f}) &\implies& \sigma \notin \Sigma_{\infty}\\
\pi\in\cl(\Pi_{f}) &\implies& \pi \notin \Pi_{\infty}.
\end{eqnarray*}
\end{remark}
\begin{proof}
\pf\
We prove the first statement. By Corollary 1 from \citet{canovas2005distance}, if $\sigma\in \Sigma_{f}$ then $\delta^{\Sigma}(\sigma,\Sigma_{i})=\delta^{\Sigma}(\sigma,\bd(\Sigma_{f}))$. Remark 2 from \citet{canovas2005distance} states $+\infty>\delta^{\Sigma}(\sigma,\Sigma_{i})=\delta^{\Sigma}(\sigma,\bd(\Sigma_{f}))$. Thus, by definition, $\sigma\notin \Sigma_{\infty}$. If $\sigma\in \bd(\Sigma_{f})$, then by definition $\sigma\notin \Sigma_{\infty}$.

We prove the second statement. If $\pi:=(c,\sigma)\in\Pi_{f}$, by definition, $\sigma\in\Sigma_{f}$. By the first statement, $\sigma\notin\Sigma_{\infty}$. By definition,
\begin{equation*}
\delta^{\Pi}(\pi,\bd(\Pi_{f})):=\inf_{\pi'\in\bd(\Pi_{f})}\delta^{\Pi}(\pi,\pi')=\inf_{\pi':=(c',\sigma')\in\bd(\Pi_{f})}\max\{\|c-c'\|, \delta^{\Sigma}(\sigma,\sigma') \}.
\end{equation*}
Consider the restricted subset of points in $\bd(\Pi_{f})$ such that $\pi'=(c,\sigma')$, i.e. the problems with all the \emph{same} cost function as $\pi$, which we denote $\bd(\Pi_{f})\cap \{ \pi' \mbox{ : } c'=c \}$. Then:
\begin{equation}\label{eq:problem_distance}
\delta^{\Pi}(\pi,\bd(\Pi_{f})) = \inf_{\pi':=(c',\sigma')\in\bd(\Pi_{f})}\max\{\|c-c'\|, \delta^{\Sigma}(\sigma,\sigma') \} \leq \inf_{\pi'\in \bd(\Pi_{f})\cap \{ \pi' \mbox{ \tiny{:} } c'=c \}} \max\{ 0 , \delta^{\Sigma}(\sigma,\sigma') \}.
\end{equation}
The last term in \eqref{eq:problem_distance} is the same as $\delta^{\Sigma}(\sigma,\bd(\Sigma_{f}))$. Since $\delta^{\Sigma}(\sigma,\bd(\Sigma_{f}))<+\infty$, it follows that $\delta^{\Pi}(\pi,\bd(\Pi_{f}))<+\infty$. Thus, by definition, $\pi\notin\Pi_{\infty}$. If $\pi\in \bd(\Pi_{f})$, then by definition $\pi\notin \Pi_{\infty}$.
 \hfill $\square$
\end{proof}



\section{Proofs of lemmas in Section \ref{sec:Lipschitz_constant_invariance}.}\label{app:proof_clairfying_lemmas}


\subsection{Proof of Lemma \ref{lem:distIllPosed}}\label{app:proof_Lemma_dist_Ill_Posed}

\distillposed*

\begin{remark}
Note that in \cite{canovas2005distance}, $\Sigma_{s}$ refers to ``strongly infeasible'', which conflicts with the notation of $\Pi_{s}$ which refers to ``solvable''. For the remainder of this proof, we will use $\Sigma_{si}$ to denote strongly infeasible to distinguish it from $\Pi_{s}$. The set of strongly infeasible problems refers to the systems which contain a finite subset of constraints that is infeasible. With this definition, it is clear that $\Sigma_{si}$ is non-empty, because the constraint system with $\langle 0_{n},x\rangle \geq 1$ as a constraint is strongly infeasible.  For brevity, write $H(\sigma_{0})$ as $H$.
\end{remark}

\begin{proof} \pf\ All we need to show is that $\delta^{\Sigma}(\sigma_{0},\Sigma_{i}) = \delta^{\Sigma}(\sigma_{0},\bd(\Sigma_{si}))$, because $\delta^{\Sigma}(\sigma_{0},\bd(\Sigma_{si})) = d(0_{n+1},\bd(H))$ follows from Theorem 6 in \citet{canovas2005distance}.

First, we show $\delta^{\Sigma}(\sigma_{0},\Sigma_{i}) \le \delta^{\Sigma}(\sigma_{0},\bd(\Sigma_{si}))$.  Since $\Sigma_{si}\subseteq\Sigma_{i}$,
\begin{equation}\label{eq:i_leq_si}
\delta^{\Sigma}(\sigma_{0},\Sigma_{i})\leq \delta^{\Sigma}(\sigma_{0},\Sigma_{si}).
\end{equation}
By Corollary 1 in \citet{canovas2005distance}, since $\sigma_{0}\notin \Sigma_{si}\subsetneq \Sigma$,
\begin{equation}\label{eq:si_equal_bdsi}
\delta^{\Sigma}(\sigma_{0},\Sigma_{si})=\delta^{\Sigma}(\sigma_{0},\bd(\Sigma_{si})).
\end{equation}
Thus,
\begin{equation}
\delta^{\Sigma}(\sigma_{0},\Sigma_{i}) \le \delta^{\Sigma}(\sigma_{0},\Sigma_{si}) = \delta^{\Sigma}(\sigma_{0},\bd(\Sigma_{si})).
\end{equation}
To prove the reverse inequality, first note that by Corollary 1 in \citet{canovas2005distance},
\begin{equation}\label{eq:bd_C_equal_i}
\delta^{\Sigma}(\sigma_{0},\Sigma_{i}) = \delta^{\Sigma}(\sigma_{0},\bd(\Sigma_{f})).
\end{equation}
%
%
%
%
%
%
For all $\sigma\in\bd(\Sigma_{f})$, we have $\sigma\notin \Sigma_{\infty}$ by definition. Thus, by Theorem 5 part (iii) from \citet{canovas2005distance}, for all $\sigma\in\bd(\Sigma_{f})$, we have $\sigma\in\bd(\Sigma_{si})$. This means that
\begin{equation}\label{eq:bd_sigmaF_greater_bd_sigmaSi}
\delta^{\Sigma}(\sigma_{0},\bd(\Sigma_{f})) \geq \delta^{\Sigma}(\sigma_{0},\bd(\Sigma_{si})).
\end{equation}
%
%

\hfill $\square$

\end{proof}


\begin{remark} This lemma gives us a way to calculate the quantity $\delta^{\Sigma}(\sigma_{0},\Sigma_{i})$ in a less abstract space.
\end{remark}

\subsection{Proof of Lemma \ref{lem:greaterThanZero}.}\label{app:proof_Lemma_greater_than_zero}

\constantDefinition*

\greaterThanZero*

We prove each of the points individually:
\begin{proof} \pf\

\begin{enumerate}
\item The required statement is equivalent to proving $0_{n+1}\notin \bd(H(\sigma_{0}))$. By Remark \ref{rem:distance_to_Pi_infty}, $\pi_{0}\notin \Pi_{\infty}$, since $\pi_{0}\in\interior(\Pi_{f})$.  Proposition 2(ii) from \citet{canovas2006distance} states that if $\pi_{0}\in\Pi\backslash\Pi_{\infty}$, then $0_{n+1}\in \ext(H(\sigma_{0}))$ if and only if $\pi_{0}\in\interior(\Pi_{f})$. Since $\pi_{0}\in\interior(\Pi_{s})\subseteq \interior(\Pi_{f})$, we immediately have $0_{n+1}\in \ext(H(\sigma_{0}))$.\\

\item The required statement is equivalent to $0_{n}\notin\bd(Z^{-}(\pi_{0}))$. Proposition 3(i) from \citet{canovas2006distance} gives $0_{n+1}\in\interior(Z^{-}(\pi_{0}))$ if and only if $\pi_{0}\in\interior(\Pi_{s})$.\\

\item Theorem 2 from \citet{canovas2006distance} states that if $\pi\in \cl(\Pi_{s})$, then $\delta^{\Pi}(\pi_{0},\bd(\Pi_{s}))=\min\{d(0_{n+1},\bd(H(\sigma_{0}))),d(0_{n},\bd(Z^{-}(\pi_{0}))) \}$. Note that $\pi\in\interior(\Pi_{s})\subseteq\cl(\Pi_{s})$. Thus, by the previous steps, $\delta^{\Pi}(\pi_{0},\bd(\Pi_{s}))>0$.\\

\item Let $\epsilon$ be such that $0 < \epsilon < \delta^{\Pi}(\pi_{0},\bd(\Pi_{s}))$ be given. By assumption, $\pi_{0}\in\Pi_{f}$, so that $\sigma_{0}\in\Sigma_{f}$. Then, by Theorem 2 from \citet{canovas2006distance}, $d(0_{n+1},\bd(H(\sigma_{0}))) \geq \delta^{\Pi}(\pi_{0},\bd(\Pi_{s}))$, noting that $\pi_{0}\in\interior(\Pi_{s})\subseteq\cl(\Pi_{s})$. Lastly, from Lemma \ref{lem:distIllPosed}, we have  $\delta^{\Sigma}(\sigma_{0},\Sigma_{i})= d(0_{n+1},\bd(H(\sigma_{0}))) \geq  \delta^{\Pi}(\pi_{0},\bd(\Pi_{s}))>\epsilon$.\\

\item For the ``in particular'' statement, we need only check that the denominators are non-zero and that $\nu(\pi_{0})$ is finite. Since we are using the Euclidean norm, $d_{*}(0_{n},\bd(Z^{-}(\pi_{0}))) = d(0_{n},\bd(Z^{-}(\pi_{0})))$. Now, $d(0_{n},\bd(Z^{-}(\pi_{0}))) \geq \delta^{\Pi}(\pi_{0},\bd(\Pi_{s})) >\epsilon > 0$ due Theorem 2 from \citet{canovas2006distance}. Also, $\delta^{\Sigma}(\sigma_{0},\Sigma_{i})>\epsilon>0$ due to the fourth statement. Lastly, since $\pi_{0}\in\interior(\Pi_{s})$, the optimal value $\nu(\pi_{0})$ is finite. \hfill $\square$

\end{enumerate}
\end{proof}

\section{Proof of Theorem \ref{thm:Lipschitz_multiple_coeff_cost_RHS_unc}.} \label{app:Proof_Theorem_Lipschitz_Multiple}
\LipschitzMultipleRO*

\begin{proof}
\pf\ Let $T:=I\times \mathbb{R}^{n}\times \mathbb{R}$. Denote the elements $\mb{t}\in T$ as tuples $\mb{t}:=(\alpha,t,s)$ where $\alpha\in I$, $t\in\mathbb{R}^{n}$ and $s\in\mathbb{R}$.

\begin{enumerate}

\item \emph{Write $\RO(\natmbU)$ as the LSIO problem $\pi_{\natmbU}$.}
First, we show that $\pi_{\natmbU}$ and $\RO(\natmbU)$ are equivalent. Note that the cost functions for the two problems are the same, and that $\pi_{\natmbU}$ contains all the constraints of $\RO(\natmbU)$ plus some additional trivial constraints. Thus $\pi_{\natmbU}$ and $\RO(\natmbU)$ have the same feasible solution set, optimal value, and optimal solution sets. Furthermore, since $\RO(\natmbU)$ satisfies the strong Slater condition with Slater constant $\rho$, and since $\langle 0_{n},x\rangle \geq -\rho +\rho = 0$ for all $x$, it follows that $\pi_{\natmbU}$ satisfies the strong Slater condition with Slater constant $\rho$.

Second, we show that $\pi_{\natmbU}\in\interior(\Pi_{s})$. By assumption, $F^{\opt}(\RO(\natmbU))$ is non-empty and bounded, so $F^{\opt}(\pi_{\natmbU})$ is non-empty and bounded. Since $\pi_{\natmbU}$ satisfies the strong Slater condition, by Theorem \ref{thm:lsc_SS_interior}, we have $\pi_{\natmbU}\in\interior(\Pi_{f})$. Thus, by Proposition 1 part (vi) from \citet{canovas2006ill}, $c\in\interior(\mbox{cone}(\{ a_{\mb{t}}\mbox{ : } \mb{t}\in T \}))$, and by part (vii), this implies that $\pi_{\natmbU}\in\interior(\Pi_{s})$.

\item \emph{Choose $\natmbV$ in an $\epsilon$-neighborhood of $\natmbU$.}
Let $\epsilon>0$ be given satisfying $0 < \epsilon< \delta^{\Pi}(\pi_{\natmbU},\bd(\Pi_{s}))$. Such an $\epsilon$ exists because $\delta^{\Pi}(\pi_{\natmbU},\bd(\Pi_{s}))>0$ by Lemma \ref{lem:greaterThanZero}. Let $\natmbV$ be any compact and convex set in $\mathbb{R}^n$ satisfying $d_{\natural}(\natmbU,\natmbV)\leq \epsilon < \delta^{\Pi}(\pi_{\natmbU},\bd(\Pi_{s}))$. Such a $\natmbV$ exists; e.g., the choice $V_{\alpha}:=U_{\alpha}+B(0,\epsilon/2)$ for all $\alpha\in I$ satisfies $d_{\natural}(\natmbU,\natmbV)<\epsilon$.

\item \emph{Define $\pi_{\natmbU;\natmbV}:=(c,\sigma_{\natmbU;\natmbV})$.} Define the LSIO problem $\pi_{\natmbU;\natmbV}:=(c,\sigma_{\natmbU;\natmbV})$ as in Theorem \ref{thm:ROLSIOtransformMultiple}. By the same theorem, $\pi_{\natmbU;\natmbV}$ is well-defined. Since for each $\alpha\in I$, $(t,s)\in V_{\alpha}\backslash U_{\alpha}$, $\argmin_{(u_{a},u_{b})\in U_{\alpha}}d((u_{a},u_{b}),(t,s))$ is an element of $\natmbU$, every constraint in $\pi_{\natmbU;\natmbV}$ is in $\pi_{\natmbU}$ and vice versa.  Thus, $\pi_{\natmbU;\natmbV}\sim_{\Pi}\pi_{\natmbU}$ Furthermore, since $\pi_{\natmbU;\natmbV}\sim_{\Pi}\pi_{\natmbU}$ and $\pi_{\natmbU}$ satisfies the strong Slater condition with the Slater constant $\rho$, it follows that $\pi_{\natmbU;\natmbV}$ also satisfies the strong Slater condition with constant $\rho>0$.

\item  \emph{Define $\pi_{\natmbV;\natmbU}:=(c,\sigma_{\natmbV;\natmbU})$.} Define the LSIO problem $\pi_{\natmbV;\natmbU}:=(c,\sigma_{\natmbV;\natmbU})$ as in Theorem \ref{thm:ROLSIOtransformMultiple}. By the same theorem, $\pi_{\natmbV;\natmbU}$ is well-defined and is equivalent to $\RO(\natmbV)$.

\item \emph{By Theorem \ref{thm:ROLSIOtransformMultiple}, $\delta^{\Pi}(\pi_{\natmbU;\natmbV},\pi_{\natmbV;\natmbU}) = \natD(\natmbU,\natmbV):= \sup_{\alpha\in I}d_{H}(U_{\alpha},V_{\alpha})<\epsilon$.}

\item \emph{Apply Theorem \ref{thm:LipschitzContinuityLSIO} with $\pi_{\natmbU;\natmbV}\mapsto \pi_{0}, \pi_{1}$ and $\pi_{\natmbV;\natmbU}\mapsto \pi_{2}$.} We check that the assumptions of Theorem \ref{thm:LipschitzContinuityLSIO} are satisfied:
\begin{enumerate}
\item \emph{$\pi_{\natmbU;\natmbV}\in \interior(\Pi_{s})$.} This proof is identical to the proof that $\pi_{\natmbU}\in\interior(\Pi_{s})$, given in Step 1.

\item \emph{$\epsilon < \delta^{\Pi}(\pi_{\natmbU;\natmbV},\bd(\Pi_{s}))$.} Recall that $\pi_{\natmbU;\natmbV}$ and $\pi_{\natmbU}$ are both in $\interior(\Pi_{s})$, have non-empty bounded optimal solution sets and are $\Pi$-equivalent to each other. Thus, we have  $\epsilon < \delta^{\Pi}(\pi_{\natmbU},\bd(\Pi_{s})) =\delta^{\Pi}(\pi_{\natmbU;\natmbV},\bd(\Pi_{s}))$, where the inequality is by assumption and the equality is by Lemma \ref{lem:LipschitzInv}.

\item \emph{$\delta^{\Pi}(\pi_{\natmbU;\natmbV},\pi_{\natmbV;\natmbU})\leq\epsilon$ and $\delta^{\Pi}(\pi_{\natmbU;\natmbV},\pi_{\natmbU;\natmbV})=0\leq\epsilon$.} The first inequality comes from Step 5, and the second inequality is trivial.
\end{enumerate}

Thus, the assumptions of Theorem \ref{thm:LipschitzContinuityLSIO} are satisfied, so:
\begin{equation}\label{eq:paper_proof_Lipschitz_pre_single_RO}
| \nu(\pi_{\natmbU;\natmbV}) - \nu(\pi_{\natmbV;\natmbU}) | \leq L(\pi_{\natmbU;\natmbV},\epsilon) \delta^{\Pi}(\pi_{\natmbU;\natmbV},\pi_{\natmbV;\natmbU}),
\end{equation}
where $L(\pi_{\natmbU;\natmbV},\epsilon)$, as defined in Definition \ref{def:constantDefinition}.

\item \emph{Lastly, show that the Lipschitz constant is independent of the choice of $V$.} 
Applying Lemma \ref{lem:LipschitzInv} to $\Pi$-equivalent problems $\pi_{\natmbU}$ and $\pi_{\natmbU;\natmbV}$, it follows that $L(\pi_{\natmbU},\epsilon) = L(\pi_{\natmbU;\natmbV},\epsilon)$. Now define $L(\natmbU,\epsilon):= L(\pi_{\natmbU},\epsilon)$. By recognizing that $\nu(\pi_{\natmbU;\natmbV})=\nu(\RO(\natmbU))$ and $\nu(\pi_{\natmbV;\natmbU})=\nu(\RO(\natmbV))$, and using $\delta^{\Pi}(\pi_{\natmbU;\natmbV},\pi_{\natmbV;\natmbU}) = d_{H}(\natmbU,\natmbV)$, we obtain the final result:
\begin{align*}
| \nu(\mb{RO}(\natmbU)) - \nu(\mb{RO}(\natmbV)) | & = | \nu(\pi_{\natmbU;\natmbV}) - \nu(\pi_{\natmbV;\natmbU}) |,\\
& \hspace{7mm} \text{(By definition of $\pi_{\natmbU;\natmbV}$ and $\pi_{\natmbV;\natmbU}$)}\\
& \leq L(\pi_{\natmbU;\natmbV},\epsilon) \delta^{\Pi}(\pi_{\natmbU;\natmbV},\pi_{\natmbV;\natmbU}),\\
& \hspace{7mm} \text{(By Theorem \ref{thm:LipschitzContinuityLSIO})}\\
& = L(\pi_{\natmbU;\natmbV},\epsilon) d_{H}(\natmbU,\natmbV),\\
& \hspace{7mm} \text{(By Theorem \ref{thm:ROLSIOtransformMultiple})} \\
& = L(\natmbU,\epsilon) d_{H}(\natmbU,\natmbV).\\
& \hspace{7mm} \text{(Since $\pi_{\natmbU;\natmbV}\sim_{\Pi}\pi_{\natmbU}$)}\\
\end{align*}
\hfill $\square$
\end{enumerate}

\end{proof}

\section{Summary of Notation}\label{app:Summary_Notation}

Notation used in this paper is summarized in Table \ref{tab:notationConstants}.

\begin{table}[H]
\centering
\scalebox{0.75}{
\begin{tabular}{l|l|l}
 \multicolumn{1}{c|}{\textbf{Notation}} & \multicolumn{1}{c|}{\textbf{Name}} & \multicolumn{1}{c}{\textbf{Definition}}\\
 \hline
$c$ & cost vector & $c\in\mathbb{R}^{n}$\\
$a$ & coefficient vector & $a:T\to\mathbb{R}^{n}$\\
$b$ & right-hand side & $b: T\to \mathbb{R}$\\
\hline
$T$ & index of LSIO problem & \\
$\pi$ & LSIO problem & $\pi:=(c,(a,b))$\\
$\sigma$ & constraint system & $\sigma:=(a,b)$\\
$\Pi$ & Set of LSIO problems & $\Pi:=\mathbb{R}^{n}\times(\mathbb{R}^{n}\times\mathbb{R})^{T}$\\
$\Sigma$ & Set of constraint systems & $\Sigma:=(\mathbb{R}^{n}\times\mathbb{R})^{T}$\\
\hline
$d(x,y)$ & Dist. b/w points in $\mathbb{R}^{n}$ & \\
$d(x,C)$ & Dist. b/w a point and a set & $\inf_{y\in C}d(x,y)$  \\
\hline
 $e(C,D)$ & Excess of $C$ on $D$ & $\sup_{x\in C}\inf_{y\in D}d(x,y)$\\
$d_{H}(C,D)$ & Hausdorff distance & $\max\left\{ \adjustlimits\sup_{u\in U}\inf_{v\in V} \|u-v\|, \adjustlimits\sup_{v\in V}\inf_{u\in U} \|u-v\|  \right\}$. \\
$\mb{d}_{r}(C,D)$ & & $ \max_{\|x\|\leq \rho}\left| d(x,C) - d(x,D) \right| $\\
$\wh{\mb{d}}_{r}(C,D)$ & & $\inf\left\{  \eta\geq 0 \left| \begin{aligned} C\cap\rho B &\subseteq D+\eta B \\ D\cap\rho B &\subseteq C+\eta B \end{aligned} \right. \right\}:=\max\{ e((C)_{r},D), e((D)_{r},C) \}$.\\
$\natD(\natmbU,\natmbV)$ &  & $\sup_{\alpha\in I} d_{H}(U_{\alpha},V_{\alpha})$ \\

\hline
$\delta_{r}(f_{1},f_{2})$ & & $\mb{d}_{r}(\mathrm{epi} f_{1}, \mathrm{epi} f_{2})$\\

$\wh{\delta}_{r}(f_{1},f_{2})$ & & $\wh{\mb{d}}_{r}(\mathrm{epi} f_{1}, \mathrm{epi} f_{2})$\\

$\wh{\delta}^{+}_{r}(f_{1},f_{2})$ & & $\inf\left\{\eta \geq 0 \left| \begin{aligned}
\min_{B(x,\eta)}f_{2} &\leq \max\{f_{1}(x),-\rho \}+\eta \\ \min_{B(x,\eta)}f_{1} & \leq \max\{f_{2}(x)-\rho \}+\eta  \mbox{ }
\end{aligned}\right. \forall x\in\rho B
\right\}$\\

$\delta^{\Sigma}(\sigma_{1},\sigma_{2})$ & distance b/w $\sigma_{1}$ and $\sigma_{2}$ & $ \delta^{\Sigma}(\sigma_{1},\sigma_{2}):=\sup_{t\in T} \left\| \left(\begin{array}{c} a^{1}(t) \\ b^{1}(t) \end{array} \right) - \left(\begin{array}{c} a^{2}(t) \\ b^{2}(t) \end{array} \right) \right\|$\\

$\delta^{\Pi}(\pi_{1},\pi_{2})$ & distance b/w $\pi_{1}$ and $\pi_{2}$ & $\delta^{\Pi}(\pi_{1},\pi_{2}):= \max\{ \|c^{1}-c^{2}\|, \delta^{\Sigma}(\sigma_{1},\sigma_{2})  \}$\\

$F(\sigma)$ & feasible solution set of $\pi$ &  $F(\sigma) := \{ x\in \mathbb{R}^{n} \mbox{ : } \langle a_{t}, x \rangle \geq b_{t}, \mbox{ } \forall t\in T\}$\\
$\nu(\pi)$ & optimal value of $\pi$ & $\nu(\pi) := \inf_{x\in\mathbb{R}^{n}} \{ \langle c, x \rangle \mbox{ : } x \in F(\sigma) \}$\\
$F^{\opt}(\pi) $ & optimal solution set of $\pi$ & $F^{\opt}(\pi) := \{ x \in \mathbb{R}^{n} \mbox{ : } \langle c , x \rangle = \nu(\pi) \}$\\
$A(\pi)$ & & $A(\pi) := \mathrm{conv}( \{ a_{t} , t\in T \} )$\\
$R(\pi)$ & & $R(\pi) := \{ -b_{t} , t\in T ; \nu(\pi) \}$\\
$Z^{-}(\pi)$ & & $Z^{-}(\pi) := \mathrm{conv}( \{ a_{t} , t\in T ; -c\} )$\\
$H(\sigma)$ & & $H(\sigma) := \mathrm{conv}\left( \left\{ \left( \begin{array}{c} a_{t} \\ b_{t} \end{array} \right)  \right\}_{t\in T} \right) + \left\{ \left( \begin{array}{c} 0_{n} \\ -\mu \end{array} \right) \right\}_{\mu \geq 0}$\\
\hline
$\Sigma_{f}$ & set of feasible $\sigma$ & $\Sigma_{f}:= \{\sigma \in \Sigma \mbox{ : } F(\sigma) \neq \varnothing \}$\\
$\Sigma_{i}$ & set of infeasible $\sigma$ & $\Sigma_{i}:= \Sigma \backslash \Sigma_{f}$\\
$\Sigma_{\infty}$ & &  $\Sigma_{\infty}:=\{\sigma\in \Sigma \mbox{ : }\delta^{\Sigma}(\sigma,\bd(\Sigma_{f}))=+\infty\}$\\
$\Pi_{f}$ & set of feasible $\pi$ & $\Pi_{f} := \{ \pi= (c,\sigma) \in \Pi \mbox{ : } \sigma \in \Sigma_{f}\}$\\
$\Pi_{i}$ & set of infeasible  $\pi$  & $\Pi_{i}:=\Pi\backslash\Pi_{f}$\\
$\Pi_{\infty}$ & &  $\Pi_{\infty}:=\{\pi\in \Pi \mbox{ : }\delta^{\Pi}(\pi,\bd(\Pi_{f}))=+\infty\}$\\
$\Pi_{s}$ & set of solvable  $\pi$  & $\Pi_{s}:= \{ \pi= (c,\sigma) \in \Pi \mbox{ : } F^{\opt}(\pi)\neq \varnothing \}$\\
\hline
$\RO(U)$ & RO with unc set $U$ & $\RO(U):=(c,U,b)$\\
$F(\RO(U))$ & feas set of RO problem & \\
$\nu(\RO(U))$ & opt value of RO problem & \\
$F^{\opt}(\RO(U))$ & opt set of RO problem & \\
$\varphi(\lambda)$ && $\varphi(\lambda)=\varphi_{*}(\lambda) := \sqrt{1+\lambda^{2}}$\\
$\psi(\alpha)$ && $\psi(\alpha) := (1+\alpha)\sqrt{1+\alpha^{2}}$\\
$\widehat{\rho}(\pi_{0})$ && $\widehat{\rho}(\pi_{0}) := \frac{\sup R(\pi_{0}) }{ d_{*}(0_{n},\bd(Z^{-}(\pi_{0}))) }$\\
$\beta(\pi_{0},\epsilon)$ && $\beta(\pi_{0},\epsilon) := \frac{\psi(\widehat{\rho}(\pi_{0}))}{\delta^{\Sigma}(\sigma_{0},\Sigma_{i}) - \epsilon}$\\
$\gamma(\pi_{0},\epsilon)$ && $\gamma(\pi_{0},\epsilon) := \varphi_{*}(0)( \widehat{\rho}(\pi_{0}) + \beta(\pi_{0})\epsilon ) + \| c^{0} \|_{*} \beta(\pi_{0})$\\
$\mu(\pi_{0},\epsilon)$ && $\mu(\pi_{0},\epsilon) := \varphi(0)\frac{ \sup R(\pi_{0}) + \epsilon\max\{1,\gamma(\pi_{0})\}}{ d(0_{n},\bd(Z^{-}(\pi_{0}))) - \epsilon}$\\
$L(\pi_{0},\epsilon)$ & Lipschitz constant & $L(\pi_{0},\epsilon) := \varphi_{*}(0) \left( (\epsilon + \|c^{0}\|)\frac{\psi(\mu(\pi_{0},\epsilon))}{\delta^{\Sigma}(\sigma_{0},\Sigma_{i})-\epsilon} + \mu(\pi_{0},\epsilon)\right)$\\
\end{tabular}
}
\caption{Summary of notation.}  \label{tab:notationConstants}
\end{table}

\section*{Acknowledgments.} This research was funded by an Ontario Graduate Scholarship and a Natural Sciences and Engineering Research Council of Canada grant.


\bibliographystyle{plainnat}
\bibliography{RobustStability} 


\end{document}